\documentclass[12pt, fullpage, reqno]{article}
\usepackage[english]{babel} 
\usepackage[utf8]{inputenc}
\usepackage[T1]{fontenc}
\usepackage{amsmath,amsfonts,amssymb}
\usepackage{graphicx}
\usepackage{dsfont}
\usepackage{stmaryrd}
\usepackage{amssymb,mathrsfs,amsthm}
\usepackage{mathtools}
\usepackage{hyperref}
\providecommand{\norm}[1]{\lVert#1\rVert}
\newtheorem{theorem}{Theorem}

\newtheorem{lem}{Lemma}

\newtheorem{remark}{Remark}

\usepackage{graphics}
\usepackage{graphicx}
\usepackage{epsfig}
\usepackage[left=3cm,right=3cm,top=3cm,bottom=3cm]{geometry}
\numberwithin{equation}{section}
\newcommand{\R}{\mathbb R}

\begin{document}
	\noindent \textsc{\Large\bf Nonparametric relative error estimation of the regression function for censored data}\\
	
\noindent  \textsc{\bf {BOUHADJERA Feriel. $^{1,\, 2}$}} \\
\noindent  {\it $^{1}$ Universit\'e  Badji-Mokhtar, Lab. de Probabilit\'es  et  Statistique.  BP 12, 23000 Annaba, Alg\'erie.}\\
 \noindent \textsc{\bf  OULD  SA\"ID Elias.} {\sf  Corresponding author. }\\
\noindent {\it $^{2}$  Universit\'e du  Littoral C\^ote d'Opale. Lab. de Math. Pures et Appliqu\'ees. IUT de Calais. 19, rue Louis David. Calais, 62228, France.}\\
 \noindent \textsc{\bf   REMITA  Mohamed Riad.} \\
\noindent   {\it Universit\'e  Badji-Mokhtar, Lab. de Probabilit\'es  et  Statistique.  BP 12, 23000 Annaba, Alg\'erie.}\\


\vspace*{.61cm}
\noindent{ \normalsize{\bf  ABSTRACT. Let $\bf (T_i)_{i }$ be a sequence of independent identically distributed  (i.i.d.) random variables (r.v.) of interest  distributed as $\bf T$ and $\bf(X_i)_{ i }$ be a corresponding vector of covariates taking values on $\bf \mathbb{R}^d$. In censorship models the r.v. $\bf T$ is subject to random censoring by another r.v. $\bf C$. In this paper we built  a new kernel estimator based on the so-called synthetic data of the mean squared relative error for the regression function. We establish the  uniform almost sure convergence with rate over a compact set and its asymptotic normality. The asymptotic variance is explicitly given and as product we give a confidence bands. A simulation study has been conducted to comfort our theoretical results}.}\\
\vspace*{-.8cm}

\noindent {\it{Keywords}}:  Asymptotic normality.   Censored data.  Kernel estimate. Relative regression error.  Uniform almost sure convergence. V-C classes.\\
\section{Introduction}
Let $\displaystyle  (X_{i},T_{i})_{ i }$ be a $\mathbb{R}^d \times \mathbb{R}^*_+, (d\geq 1)$ valued sequence of random vectors that we assume drawn from the pair $(X,T)$ which is defined on a probability space $(\Omega, \mathcal{F},\mathbb{P})$.
The purpose of this work is to study the effect of a random covariable $X$ on a r.v. $T$ which is subject to right censoring by another r.v. $C$. This relation of regression is modeled by: 
\begin{equation}\label{model}
\displaystyle T=r(X)+\epsilon,
\end{equation}
where $r(\cdot)$ is the regression function and $\epsilon$ a sequence of error independent to $X$. Usually, $\displaystyle  r(\cdot)=\mathbb{E}[T|X]$ is estimated by minimizing the mean squared loss function $\displaystyle  \mathbb{E}[(T-r(X))^{2} \big| X]$. However, this loss function is based on some restrictive conditions that is the variance of the residual is the same for all the observations, which is inadequate when the data contains some outliers. Therefore, in order to overcome this drawback we consider an alternative approach allow to construct an efficient predictor even if the data is affected by the presence of outliers. So, in this paper the limitations of the classical regression are counteracted by estimating the regression function with respect to the minimization of the following mean squared relative error, for $T>0$,
\begin{equation}\label{reg_rel}
\displaystyle  \mathbb{E}\Big[\left(  \frac{ T-r(X) }{T}\right)^{2} \big| X\Big].
\end{equation}
The latter is a more meaningful measure of performance of a predictor than the usual error in the presence of outliers. 
It is easy  to see   that the  solution of the minimization  problem of  (\ref{reg_rel})  is given by 
 \begin{equation}\label{regrelative}
  r(X)= \frac{\mathbb{E}\big[T^{-1}|X\big]}{\mathbb{E}\big[T^{-2}|X\big]}.
  \end{equation}
 \hyperref[park1]{Park and Stefanski (1998)}  \, have shown  
that the solution given by (\ref{regrelative})  satisfies
\begin{equation}\label{ineg}
\displaystyle \frac{\mathbb{E}\big[T^{-1}|X\big]}{\mathbb{E}\big[T^{-2}|X\big]} \leq \mathbb{E}\big[T \big| X \big] \quad \quad \text{a.s.}
\end{equation}
\noindent provided that the first two conditional inverse moments are finite. The authors consider parametric approaches to estimate  the regression function $r(\cdot)$ which focused on estimating the mean and variance functions modeling methods (\hyperref[carroll]{Carroll and Ruppert, 1988}) of the inverse response $T^{-1}$ as function of $X$. Without claiming to be exhaustive, we can quote \hyperref[narula]{Narula and Wellington (1977)} who studied an estimation method for minimizing the sum of absolute relatives residuals. \hyperref[farum]{Farum (1990)} developed an estimation method designed to reduce absolute relative error. \hyperref[khoshgoftaar]{Khoshgoftaar {\it et al}. (1992)} studied the asymptotic properties of the estimators minimizing the sum of squared relative errors. In this contribution, we focus on nonparametric approach. To the best of our knowledge, only the paper of \hyperref[park]{Park {\it et al.} (2008)} study the nonparametric regression using the relative error as loss function. They studied the asymptotic properties of an estimator minimizing the sum of squared relative errors by applying local linear approach.  \\
In many estimation problems, it is not always possible, to make complete measurements when the available sample data is incomplete in the sense that measures are not available for all members of a random sample. For example, in medical follow-up studies, it often happens for various reasons, that the duration of interest can not be observed. This may be due to the loss of view of the patient at the beginning or end of the study period. These values are censored. The censored values, although unknown, must be taken into account to obtain a correct estimate and precise conclusions. For such practical observations, conventional statistical procedures are no longer valid and more elaborate techniques are used to model such observations.\\
One of the classical cases for incomplete data is the right-censored data. In this case, we observe another r.v. $C$ with continuous distribution function (d.f.) $G$, we can only observe a sample $(X_i, Y_i, \delta_i)$ where $Y_i=T_i \wedge C_i$ and $\delta_i={\mathds{1}}_{\{T_i \leq C_i\}}$, for $i=1,\dots,n$, with $\wedge$ denotes the minimum and $\mathds{1}_A$ is the indicator function of the event A.\\
When we talk about censored data, several authors like \hyperref[carbonez]{Carbonez {\it et al.} (1995)}, \hyperref[kohler]{Kohler {\it et al.} (2002)}, \hyperref[delacroix]{Delecroix {\it et al.} (2008)} and \hyperref[guessoum]{Guessoum and Ould Said (2008)} uses the synthetic data that take into account the effect of censorship on the distribution. For that we consider the sample $(Y_i,\delta_i)_i$, for  $1\leq i \leq n$ and we put :
\begin{equation}\label{calsynt}
\displaystyle  T_{i}^{\ast,-\ell}=\frac{\delta_i Y_{i}^{-\ell}}{\overline{G}(Y_i)}  ,  \quad  \quad\ell=1,2    
\end{equation}
where $\overline{G}$ is the survival function of the censoring rv $C$.\\ All along this paper, we suppose that: 
\begin{equation}\label{indep}
\displaystyle  (T_i, X_i)_i \;\text{and} \;(C_i)_i \;\text{are independent for} \; 1\leq i \leq n.
\end{equation}
Then from the equation (\ref{calsynt}) and the hypothesis (\ref{indep}), we get, 
\begin{equation*}
\begin{aligned}
\displaystyle \mathbb{E}[T_{1}^{\ast,-\ell}|X_1]&= \mathbb{E}\left[\frac{\delta_1 Y_{1}^{-\ell}}{\overline{G}(Y_1)}|X_1\right]\\
&\displaystyle  =\mathbb{E}\left\{\mathbb{E}\left[\frac{\delta_1 Y_{1}^{-\ell}}{\overline{G}(Y_1)}|T_1,X_1\right]|X_1\right\}\\
&\displaystyle  =\mathbb{E}  \left\{\frac{T_{1}^{-\ell}}{\overline{G}(T_1)} \mathbb{E} \left[ \mathds{1}_{\{T_1 \leq C_1 \}} |T_1 \right] |X_1\right\}\\
&\displaystyle  =\mathbb{E}[T_{1}^{-\ell}|X_1].
\end{aligned}
\end{equation*}
This paper offers then an alternative approach to traditional estimation models by considering the minimization of the least relative error for regressions models when the data are randomly right censored. We establish the strong and uniform consistencie (with rate) of the constructed estimator and then the asymptotic normality has been shown. At the best of our acknowledge there is no result concerning the nonparametric regression function  for censoring data  using the relative error.   \\
The rest of the paper is organized as follows: Section \ref{Sect 2} is devoted to the presentation of the new estimator of the mean squared relative error of the regression function. The assumptions and main results are given in Section \ref{Sect 3}. Simulations are drawn in Section \ref{Sect 4}. Finally, the proofs are relegated to Section \ref{Sect 5} with some auxiliary results. 
\section{Definition of the new estimator}\label{Sect 2}
Let $(T_i)_{1\leq i \leq n}$ be an i.i.d. $n$-sample of r.v. of interest with commun unknown continuous d.f. $F$ and let $(X_i)_{1\leq i \leq n}$ be a corresponding vector of covariates with joint density function $f(\cdot)$. 
As mentioned before 
 the solution of (\ref{reg_rel})  is given by 
\begin{equation}\label{P-S}
\displaystyle  r(x)=\frac{\mathbb{E}\big[T^{-1}|X=x\big]}{\mathbb{E}\big[T^{-2}|X=x\big]}:=\frac{\bar{r}_1(x)}{\bar{r}_2(x)},    
\end{equation}
with $\displaystyle r_{\ell}(\cdot)= \displaystyle \frac{\bar{r}_{\ell}(\cdot)}{f(\cdot)}$ where $\displaystyle \bar{r}_{\ell}(x)= \displaystyle \int_{\R^*_+} t^{-\ell}f_{T,X}(t,x)dt$ \; for \; $\ell=1,2$.\\ 
Recall that, in the case of complete data, a well-known Nadaraya Watson (N-W) estimator of $r(\cdot)$ is given by 
\[\displaystyle{r^*(x)=\sum_{i=1}^{n} T_{i} W_{i,n}(x)}\]
with
\[\displaystyle 
W_{i,n}(x)=
\left\{
\begin{array}{cl}
\displaystyle\frac{K\left(\frac{x-X_i}{h_n}\right)}{\sum_{i=1}^{n}K\left(\frac{x-X_i}{h_n}\right)}, & \quad \text{if} \quad \sum_{i=1}^{n}K\left(\frac{x-X_i}{h_n}\right) \neq 0; \\
1/n, & \quad \text{otherwise},
\end{array}
\right.
\]
 where $h_n$ is a sequence of positive real numbers (bandwidth) that decreases to zero when $n$ goes to infinity and $K$ is a kernel function defined in $\mathbb{R}^d$.
Thus, a natural estimator of (\ref{P-S}) is given by
\begin{equation}
\displaystyle  {r}_{n}(x)=\frac{\displaystyle{\sum_{i=1}^{n}T_{i}^{-1}K\left(\frac{x-X_i}{h_n}\right)}}{\displaystyle{\sum_{i=1}^{n}T_{i}^{-2}K\left(\frac{x-X_i}{h_n}\right)}},   
\end{equation}
this is the  analogous  N-W estimator which is nothing other than a special case of the censored case.\\
As mentioned before, when   the r.v. $T$ is subject to right censoring by another r.v. $C$,  we define $\tilde{r}_{n}(x)$ as a "pseudo-estimator" of $r(x)$ that is, for any $ \displaystyle x \in \mathbb{R}$, we have,
\begin{equation}\label{rntilde}
\displaystyle  \tilde{r}_{n}(x)=\frac{\displaystyle{\sum_{i=1}^{n}\frac{\delta_{i}Y_{i}^{-1}}{\overline{G}(Y_i)}K\left(\frac{x-X_i}{h_n}\right)}}{\displaystyle{\sum_{i=1}^{n}\frac{\delta_{i}Y_{i}^{-2}}{\overline{G}(Y_i)}K\left(\frac{x-X_i}{h_n}\right)}}=:\frac{\bar{\tilde{r}}_{1,n}(x)}{\bar{\tilde{r}}_{2,n}(x)}.    
\end{equation}
The latter can not be calculated as $\overline{G}$ is unknown. Then to define a genuine estimator of $r(\cdot)$, we replace $\overline{G}$ by its Kaplan-Meier (1958) estimator which is defined by
\begin{equation}\label{K-M}
\overline{G}_n(t)=
\left\{
\begin{array}{cl}
\displaystyle{\prod_{i=1}^{n}{\left(1-\frac{1-\delta_{i}}{n-i+1}\right)}^{\mathds{1}_{\{Y_{i}\leq t\}}}} & \quad \text{if} \quad t<Y_{(n)}, \\
0 & \quad \text{otherwise},
\end{array}
\right.
\end{equation}
where $Y_{(1)}\leq Y_{(2)}\leq \dots \leq Y_{(n)}$ are the order statistics of the $Y_{i}$ and $\delta_{i}$ is the indicator of non-censoring. The properties of $\bar{G}_n(t)$ have been studied by many authors. 
So a calculable estimator of $r(\cdot)$ is given by 
\begin{equation}\label{estimRRC}
\displaystyle  r_{n}(x)=\frac{\displaystyle{\sum_{i=1}^{n}\frac{\delta_{i}Y_{i}^{-1}}{\overline{G}_{n}(Y_i)}K\left(\frac{x-X_i}{h_n}\right)}}{\displaystyle{\sum_{i=1}^{n}\frac{\delta_{i}Y_{i}^{-2}}{\overline{G}_{n}(Y_i)}K\left(\frac{x-X_i}{h_n}\right)}}=:\frac{r_{1,n}(x)}{r_{2,n}(x)}
\end{equation}
where 
\begin{equation}\label{estimerell}
\displaystyle r_{\ell,n}(x)= \frac{\bar{r}_{\ell,n}(x)}{f_n(x)}=\frac{\displaystyle{\sum_{i=1}^{n} \frac{\delta_i Y_{i}^{-\ell}}{\overline{G}_n (Y_i)}K\left(\frac{x-X_i}{h_n}\right)}}{\displaystyle{\sum_{i=1}^{n} K\left(\frac{x-X_i}{h_n}\right)}},
\end{equation}
for $\ell=1,\; 2$ and $f_n(\cdot )$ is the well-known kernel estimator of the joint density function $f(\cdot)$.

\section{Hypotheses and main results}\label{Sect 3}

\noindent In order to state our results, we introduce some notations. For any d.f. $L$, let $\tau_{L}=sup\{y,{L}(y)<1\}$ be a upper endpoint of ${L}$. Assume that $\tau_{F}<\infty $, $\overline{G}(\tau_{F})>0$. All along the paper, when no confusion is possible, we denote by $M$ any generic strictly positive constant such that $M \geq T^{-\ell}$ and by $r_{\ell }(\cdot )=\mathbb{E}[T^{-\ell}|X=\cdot ]$ the conditional $\ell$-inverse moments of $T$ given $X$ and $\ell=1,\, 2$. Furthermore $\bar{r}_{\ell}(\cdot)=r_{\ell}(\cdot)f(\cdot)$, with $f$ is the density of $X$. On the other hand, $\log_2(\cdot)=\log\log(\cdot)$ denotes the iterated logarithm function. Finally 
 denote ${\mathcal{C}}_{0}=\left\{x \in \mathbb{R} / f(x) >0 \right\}$ the open set  and  ${\mathcal{C}}$ be a compact subset of  ${\mathcal{C}}_{0}$.\\
We will make use of the following hypotheses.
\begin{itemize}
	\item[\bf{\text{H.}}]\label{h}The bandwidth $h_n$ satisfies: 
	\begin{itemize}
		\item[$i)$] $\lim_{n_\rightarrow \infty} h_n=0, \quad \lim_{n_\rightarrow \infty} n h_n=+\infty,\quad \lim_{n\rightarrow\infty} \frac{\log n}{nh_n}=0$
		\item[$ii)$]  $h_n\log_2 n =o(1)$.
		\item[$iii)$]$\lim_{n\rightarrow\infty}nh_{n}^{5}=0$.
	\end{itemize}
	\item[\bf{\text{K.}}]\label{k} The kernel $K(\cdot)$ is: 
	\begin{itemize}
		\item[$i)$] Continuously differentiable compactly supported density function.
		\item[$ii)$]  $\int_{\mathbb{R}} |t|K(t)dt=0$ and $\int_{\mathbb{R}} t^2 K(t) dt < \infty$ holds.
		\item[$iii)$] $\int_{\mathbb{R}} t K^2(t) dt < \infty$
	\end{itemize}
	\item[\bf{\text{D.}}]\label{d}
	\begin{itemize}
		\item[$i)$] The function $\bar{r}_{\ell}(\cdot)$, for $\ell=1,2$, is twice continuously differentiable and \\ $\sup_{x \in \mathcal{C}}|\bar{r}^{''}_{\ell}(x)|<{+\infty}$.
		\item[$ii)$] The function $\Upsilon_\ell(\cdot)$, for $\ell=2,3,4$, is continuously differentiable and \\ $\sup_{x \in \mathcal{C}}{\Upsilon_\ell}^{'}(x)<{+\infty}$.
		\item[$iii)$] There exists $\Gamma>0$ such that  $\bar{r}_{2}(x)>\Gamma$ for all $x \in \mathcal{C}.$
	\end{itemize}
\end{itemize} 
\subsection{Discussions on the hypotheses}
\begin{enumerate}
	{\it
		\item The independence assumption between $(C_i)_i$ and $(T_i,X_i)_i$ may seem to
		be strong and one can think of replacing it by a classical conditional independence
		assumption between $(C_i)_i$ and $(T_i)_i$ given $(X_i)_i$. 
		However in the conditionally hypothesis we propose  the
		following estimator for the regression function $r(x)$ where  $\overline{r}_{\ell,n}$ for $\ell=1, 2$ are given by
		\begin{equation}\label{regbar}
		 \overline{r}_{\ell,n}(x) = {\displaystyle{\sum_{i=1}^{n} \frac{\delta_i Y_{i}^{-\ell}}{\overline{G}_n (Y_i \big| X_i)}K\left(\frac{x-X_i}{h_n}\right)}} 
		 \end{equation}
		where $\overline{G}_n (Y_i \big| X_i)$ is  Beran's estimator of the survival conditional distribution of the
		censored r.v. $C$ given $X$.  Then we get  an analogous estimator as in   (\ref{estimRRC}) using (\ref{regbar}).
		As mentioned before and as far as we know there is  no rate of convergence for this estimate 
		as in the unconditional case (see Deheuvels and Einmahl, 2000). We think that  this issue has to be addressed if we aim 
		to get rates of convergence.
		Moreover our framework is classical and was considered by Carbonez {\it et al.} (1995) and
		Kohler {\it et al.} (2002) among others. Note finally that this assumption implies
		the independence between $(C_i)_i$ and $(T_i)_i$ which ensures the identifiability of
		the model. \item The hypothesis $\tau_F<\tau_G$ is classical for asymptotic normality results
		in the censorship framework. It implies that $\overline{G}(T)\geq\overline{G}(\tau_F)>0\;a.s.$
		\item The Hypotheses \hyperref[h]{\bf{\text{H}}} \textit{i)} and \hyperref[k]{\bf{\text{K}}} concern the smoothing parameter $h_n$ and the kernel $K(.)$ and are standard in nonparametric regression estimation for complete or incomplete data. Moreover, \hyperref[d]{\bf{\text{D}}} \textit{i)} is needed to study the bias term. On another side, hypothesis \hyperref[d]{\bf{\text{D}}} \textit{iii)} is used to state the uniform consistency of the constructed estimator.  Finally, hypotheses \hyperref[h]{\bf{\text{H}}} \textit{ii)}, \textit{iii)} and \hyperref[d]{\bf{\textit{D}}} \textit{ii)} are needed for get asymptotic normality.
	}
\end{enumerate}
\subsection{Results}
We can now present our results. The proofs of these are established in Section \ref{Sect 5}.
We first state a uniform consistency result with rate for $r_{n}(\cdot)$.
\begin{theorem} \label{theo1}
	Under hypotheses \hyperref[h]{\bf{\textit{H}}} \textit{i)}, \hyperref[k]{\bf{\textit{K}}} and \hyperref[d]{\bf{\textit{D}}} \textit{i)}, \textit{iii)}, we have: 
	\[\sup_{x \in \mathcal{C}} \mid {r_{n}(x)-r(x)} \mid= {O}_{a.s.}\left\{\max \left(\left( \frac {\log n }{nh_{n}^{2}}\right)^{1/2},h_{n}^{2}\right)\right\} \quad as \quad n \longrightarrow \infty. \]
\end{theorem}
\begin{remark}
It is clear that we can  give the same result in $\mathbb{R}^d, \, d>1$ without difficulties. The proofs are analogous. Therefore, the \hyperref[theo1]{Theorem 1} becomes: 
\[\sup_{x \in \mathcal{C}} \mid {r_{n}(x)-r(x)} \mid= {O}_{a.s.}\left\{\max \left(\left( \frac {\log n }{n h_{n}^{d+1}}\right)^{1/2},h_{n}^{2}\right)\right\} \quad as \quad n \longrightarrow \infty. \] 
\end{remark}

\noindent In what follows we will state the asymptotic normality result. For this, let
\[
\Sigma(x)=
\left(
\begin{array}{cl}
\Upsilon_{2}(x) & \Upsilon_{3}(x) \\
\Upsilon_{3}(x) & \Upsilon_{4}(x)
\end{array}
\right)
\]
be the covariance matrix, with  
\[\Upsilon_{2\ell}(x)=\int \frac{t^{-2\ell}}{\overline{G}(t)}f_{T,X}(t,x)dt \; \; \; \; \; \text{and} \; \; \; \; \; \Upsilon_{3}(x)=\int \frac{t^{-3}}{\overline{G}(t)}f_{T,X}(t,x)dt,  \quad \quad  \text{for} \quad \ell=1,2. \]
Now we are in position to give our asymptotic normality result.
\begin{theorem}\label{theo2}
	Suppose that hypotheses \hyperref[h]{\bf{\textit{H}}}, \hyperref[k]{\bf{\textit{K}}} and \hyperref[d]{\bf{\textit{D}}}\textit{i)}, \textit{ii)} holds. Let ${\cal A}=\Big\{ x\in {\cal C}\;  \textrm{and} \; r_\ell (x) \neq 0, \; \ell=1,\, 2 \; \text{and} \; \Upsilon_j(x) \neq 0,\;j=2,3,4\Big\}$, we have
	\[\sqrt{nh_n}(r_{n}(x)-r(x))\xrightarrow{\  \mathcal{D}    \ } \mathcal{N}(0,\sigma^2(x)) \quad as \quad n \longrightarrow \infty\]
	where 
\[\sigma^2(x)=\kappa\frac{\Upsilon_2(x)\bar{r}_2^2(x)-2\Upsilon_3(x)\bar{r}_1(x)\bar{r}_2(x)+\Upsilon_4(x)\bar{r}_1^2(x)}{\bar{r}_2^4(x)}\]  
	for $\displaystyle \kappa=\int K^{2}(t)dt$ and $\xrightarrow{\  \mathcal{D}    \ }$ denotes the convergence in distribution.
\end{theorem}
\subsection{Confidence interval}
The determination of confidence interval requires the estimation of the unknown quantity $\sigma_n(x)$.
A plug-in estimate and using the following estimate of $\Upsilon_{2\ell}(x)$, for $\ell=1,2$, and  $\Upsilon_{3}(x)$ given by
\begin{equation*}\label{upsilon2l}
\displaystyle{ \widehat{\Upsilon}_{2\ell}(x)=\frac{1}{n h_n} \sum_{i=1}^n\frac{Y_i^{-2\ell}}{\overline{G}_n(Y_i)}K\left(\frac{x-X_i}{h_n}\right)}
\; \; \; \; \; \text{and} \; \; \; \; \;
\displaystyle{ \widehat{\Upsilon}_{3}(x)=\frac{1}{n h_n} \sum_{i=1}^n\frac{Y_i^{-3}}{\overline{G}_n(Y_i)}K\left(\frac{x-X_i}{h_n}\right)}
\end{equation*}
respectively and (\ref{estimRRC}) we get a consistent estimate of  $\sigma^2(x)$. This yields a confidence interval of  asymptotic level   $1-\zeta $ for $r(x)$ given by
$$\displaystyle \left[r_n(x)-t_{1-\zeta/2}\frac{\sigma_n(x)}{\sqrt{n h_n}}\; ;\; r_n(x)+t_{1-\zeta/2}{\frac{\sigma_n(x)}{\sqrt{n h_n}}}\right] $$
where $t_{1-\zeta/2}$ denotes the ${1-\zeta/2}$ quantile of the standard normal distribution.
\subsection{Comeback to complete data}
At the best of our   knowledge  there are no analogous results for the complete.The analogous results can be state by putting $C=+\infty$ and therefore $\overline{G}(\cdot)=1$. \\
To give an overview of the performance of our estimator, we graph it in the next section.
\section{Simulation Study}{\label{Sect 4}}
The main objective of this part is to evaluate the good behavior of our estimator for different censoring rates and sample sizes and to show the efficiency of this approach compared to the classical one.
\subsection{Consistency}
\subsubsection{Simulations settings}
For this purpose, simulation data are generated from model (\ref{model}) where covariates $X$ have normal distribution on ${\mathcal{N}}(5,2)$ and random effect $\epsilon$ have standard normal distribution. For the rest, we proceed in the following way:
\begin{itemize}
	\item[$\bullet$] Generate the censoring variable $C$ according to the normal law with ($\mu=11, \sigma=1$). 
	\item[$\bullet$] Calculate the response variable $T=\alpha X+\beta+c \epsilon $ with ($\alpha=2,  \beta=1$ and $c=0.2$).
	\item[$\bullet$]The censored data are calculated as $Y=T\wedge C$ and $\delta={\mathds{1}}_{\{T \leq C\}}$.The observed data therefore becomes $(X,Y,\delta)$.
	\item[$\bullet$] The Kaplan-Meier estimator is calculated for the distribution function of censorship variable $C$ in (\ref{K-M}).
	\item[$\bullet$] The choice of $K$ is not decisive, we choose then the standard Gaussian kernel ($K(u)=\frac{1}{\sqrt{2\pi}}\exp{(-\frac{u^2}{2})}$). In contrast, the choice of bandwidth is crucial that's why we take the optimal one $h_n=0.55\left(\frac{log(n)}{n}\right)^{0.33}$.
	\item[$\bullet$] Finally, we calculate the expression of our estimator obtained from (\ref{estimRRC}) for a compact set ${\mathcal{C}}=[1,4]$. 
\end{itemize}
Under each simulation setting, 100,300 and 500 replications are conducted.
\subsubsection{Simulation results}
\paragraph{Effect of sample size with fixed censorship rate.}
From Figure \ref{figure1}, we can see that the quality of fit increases with $n$ when censoring rate (CR) and bandwidth kept unchanged.
\begin{figure}[!h]
	\begin{minipage}[c]{.26\linewidth}
		\includegraphics[width=1.4\textwidth]{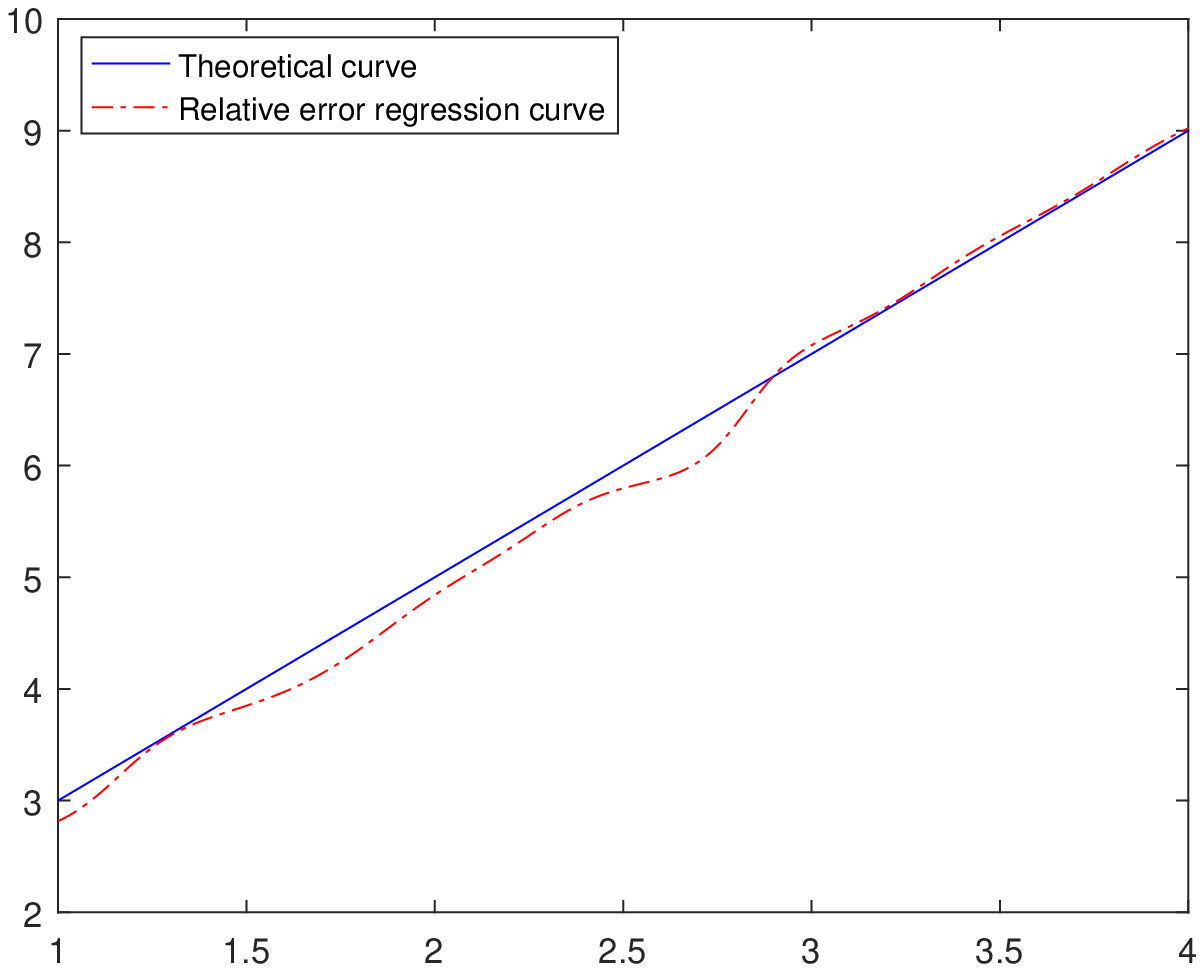}
	\end{minipage} \hfill
	\begin{minipage}[c]{.26\linewidth}
		\includegraphics[width=1.4\textwidth]{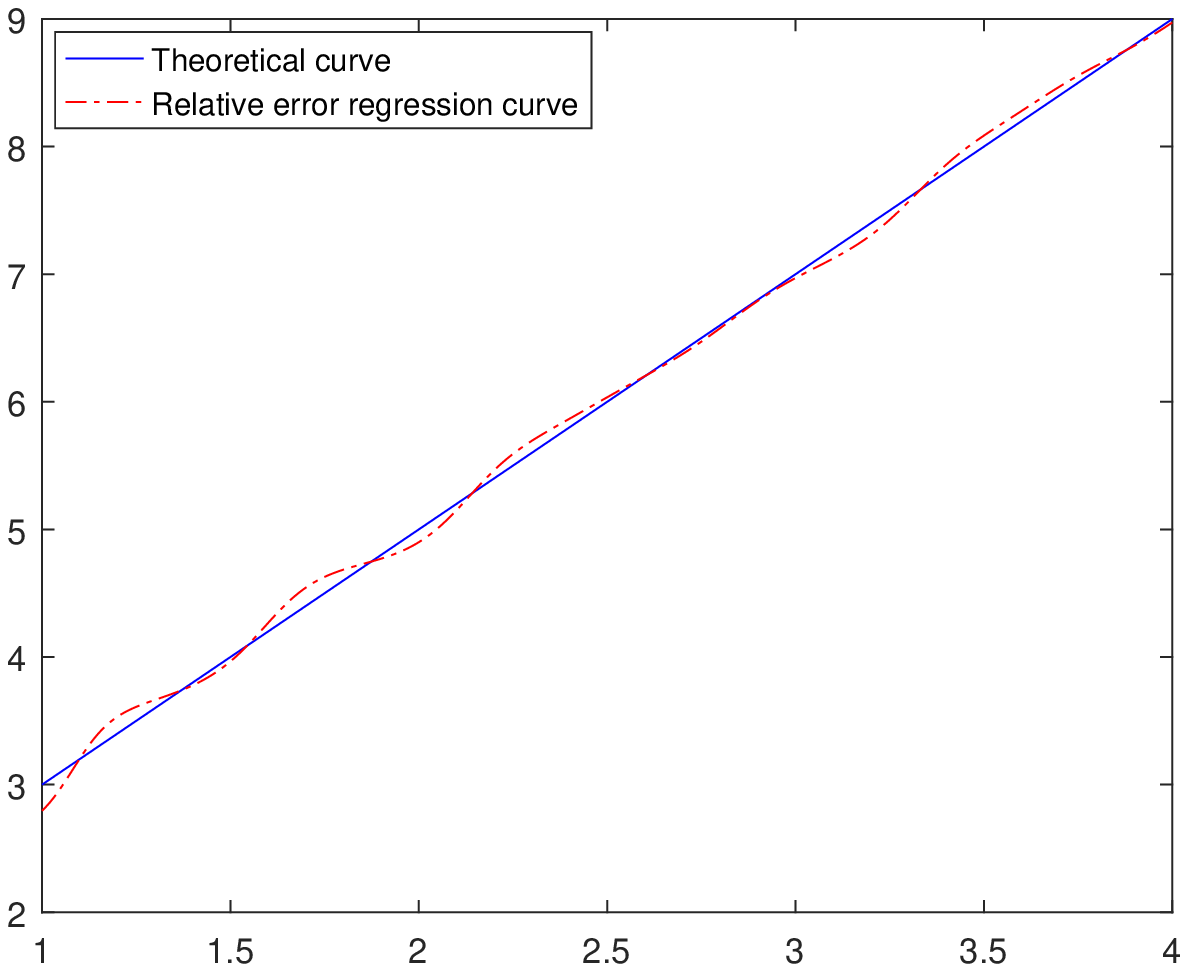}
	\end{minipage} \hfill
	\begin{minipage}[c]{.26\linewidth}
		\includegraphics[width=1.4\textwidth]{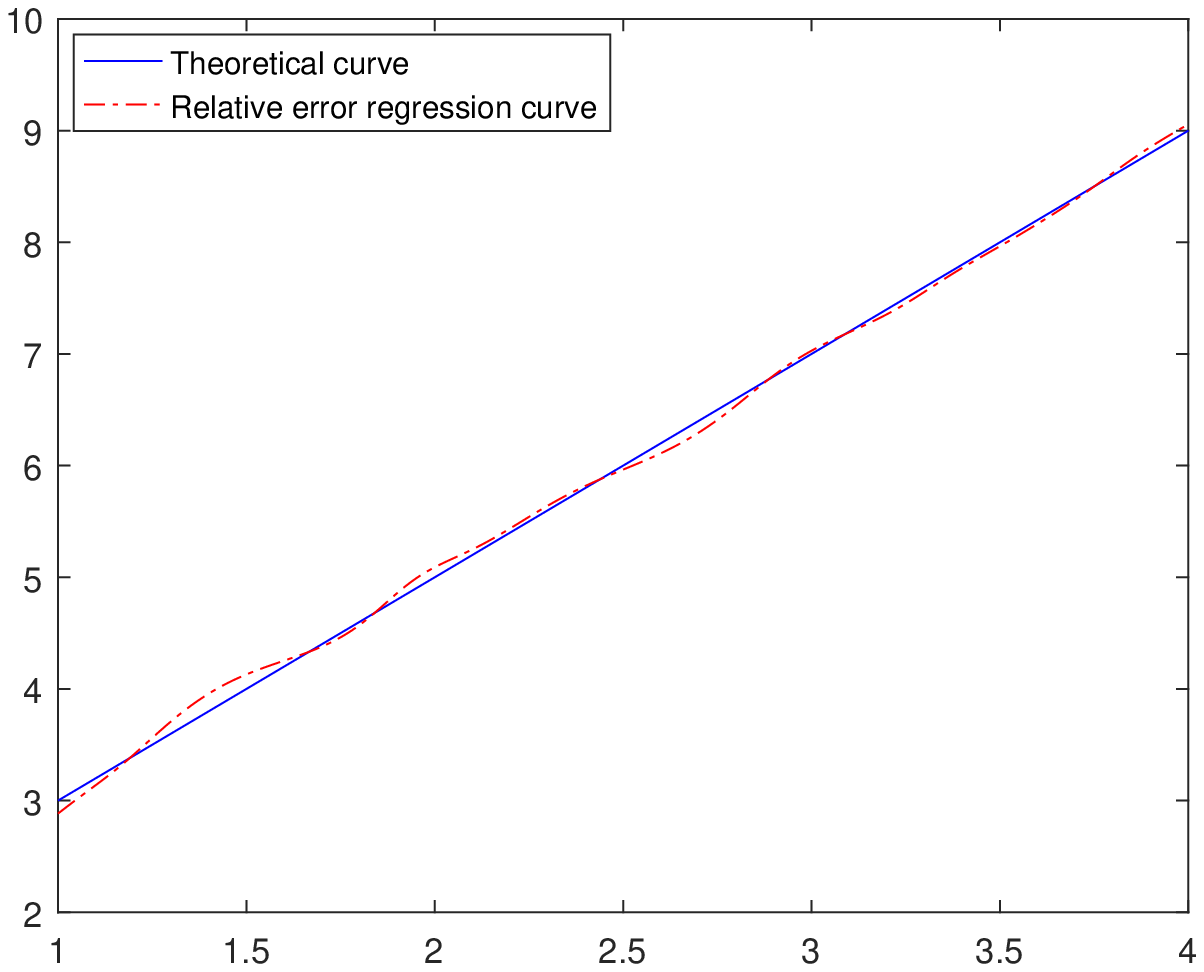}
	\end{minipage}\hfill\hfill
	\caption{$r(x)=2x+1$, $ CR\approx 50\%$ and $n=100,300$ and $500$, respectively.}\label{figure1}
\end{figure}
\paragraph{Effect of censoring rate (CR) with a fixed sample size.}
Figure \ref{figure2} is obtained by varying the censoring rate for a fixed sample size (n = 300) and for that we push the variable of interest on the right by increasing the average of the normal distribution to observe more censorship variable (the number of complete observation decreases). It can be seen that the forecasting quality decreases when the CR rises in particular in the border.
\begin{figure}[!h]
	\begin{minipage}[c]{.26\linewidth}
		\includegraphics[width=1.4\textwidth]{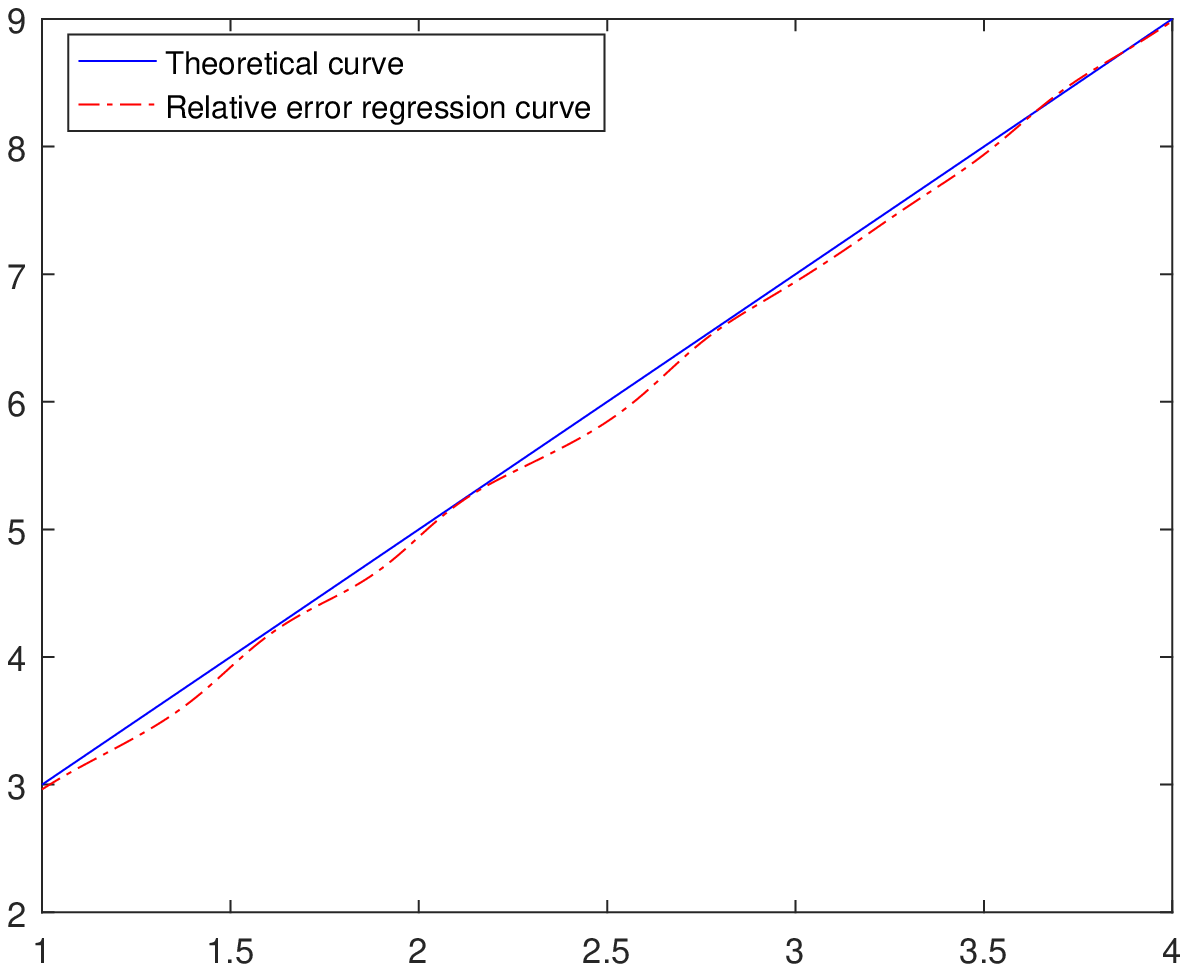}
	\end{minipage} \hfill
	\begin{minipage}[c]{.26\linewidth}
		\includegraphics[width=1.4\textwidth]{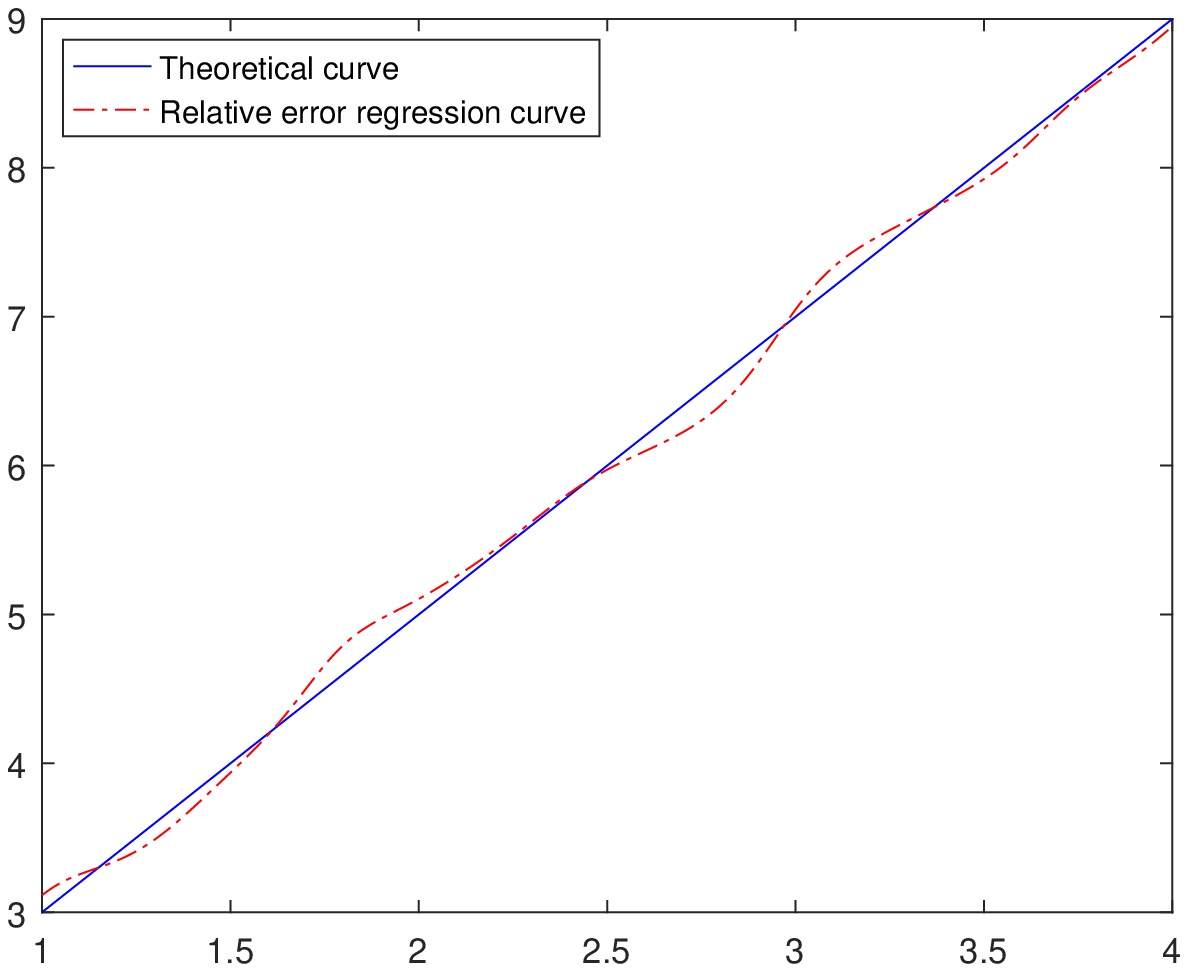}
	\end{minipage} \hfill
	\begin{minipage}[c]{.26\linewidth}
		\includegraphics[width=1.4\textwidth]{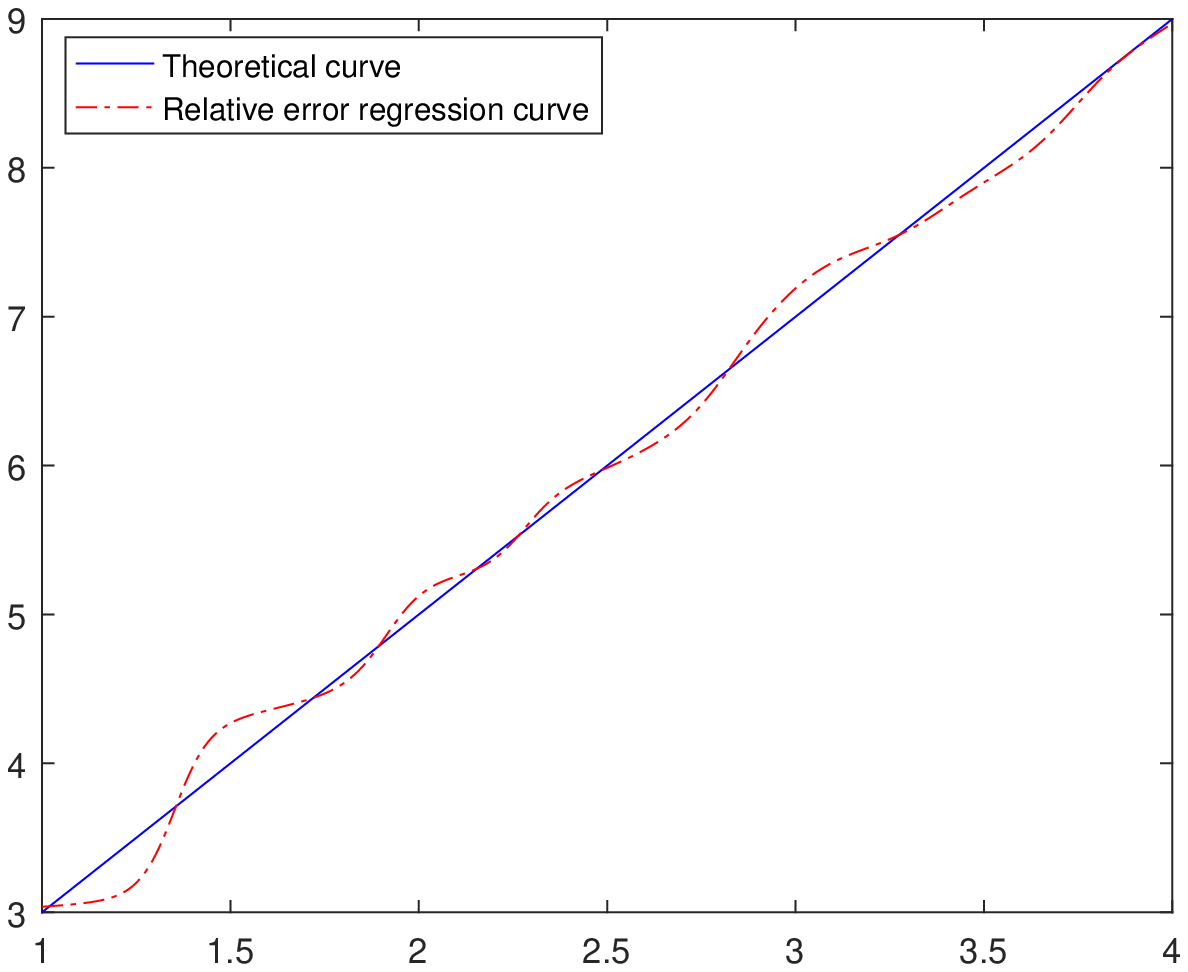}
	\end{minipage}\hfill\hfill
	\caption{ $r(x)=2x+1$, $n=300$ and $ CR\approx 15\%,50\%$ and $80\%$, respectively.}\label{figure2}
\end{figure}
\paragraph{Nonlinear functions}
We consider the case of nonlinear regression by choosing this three kinds of model:\\
(1) Parabolic $T=x^2+1+\varepsilon$, \\ (2) Sinusoidal $T=\sin\big(\frac{1}{2}x\big)^2+1+\varepsilon$, \\ (3) Exponential $T=\exp\big(\frac{1}{2}x\big)+\varepsilon$.\\
The curves are shown in Figure \ref{figure3}. Note that the quality of fit deteriorates when the period is very small.
\begin{figure}[!h]
	\begin{minipage}[c]{.26\linewidth}
		\includegraphics[width=1.4\textwidth]{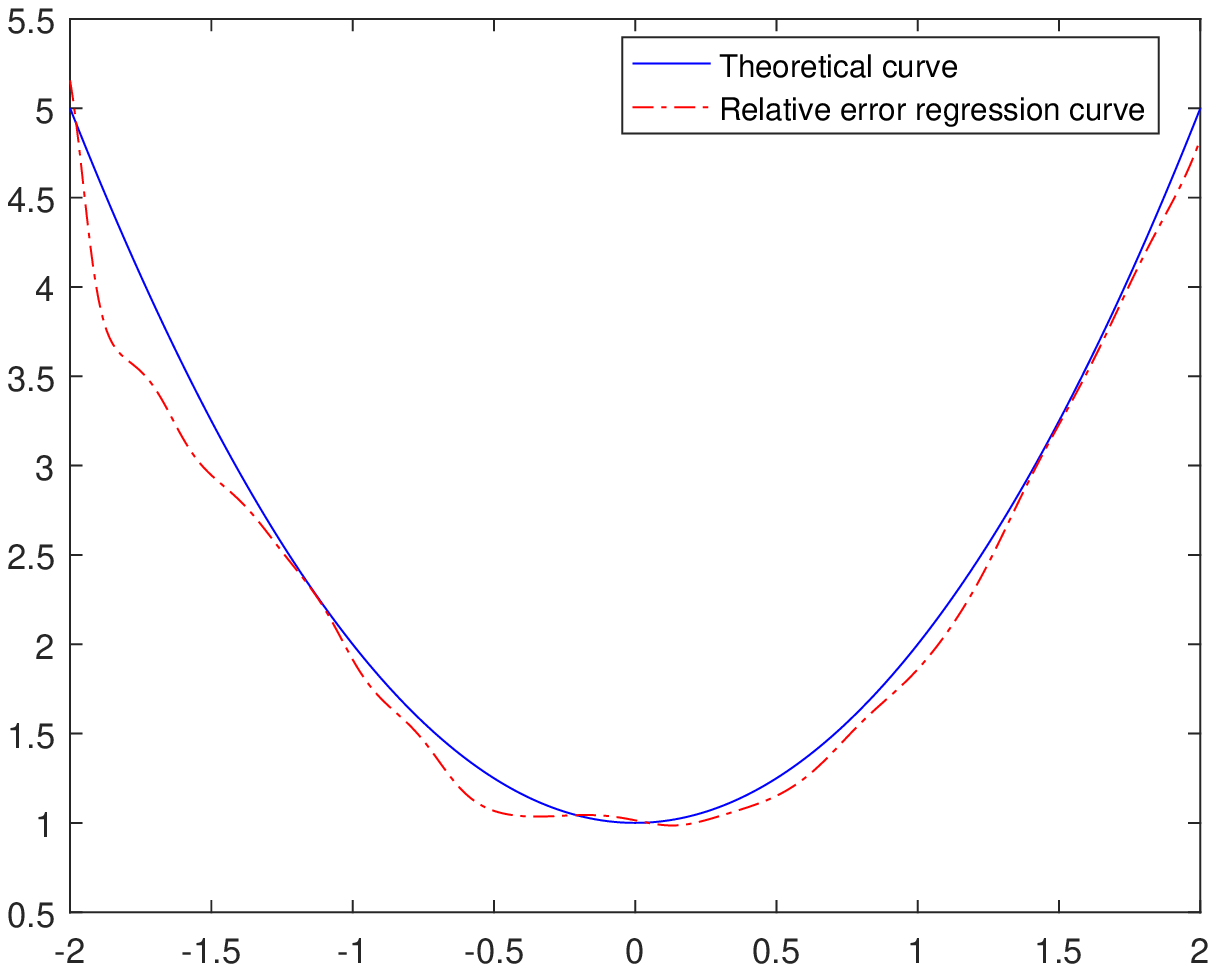}   
	\end{minipage} \hfill
	\begin{minipage}[c]{.26\linewidth}
		\includegraphics[width=1.4\textwidth]{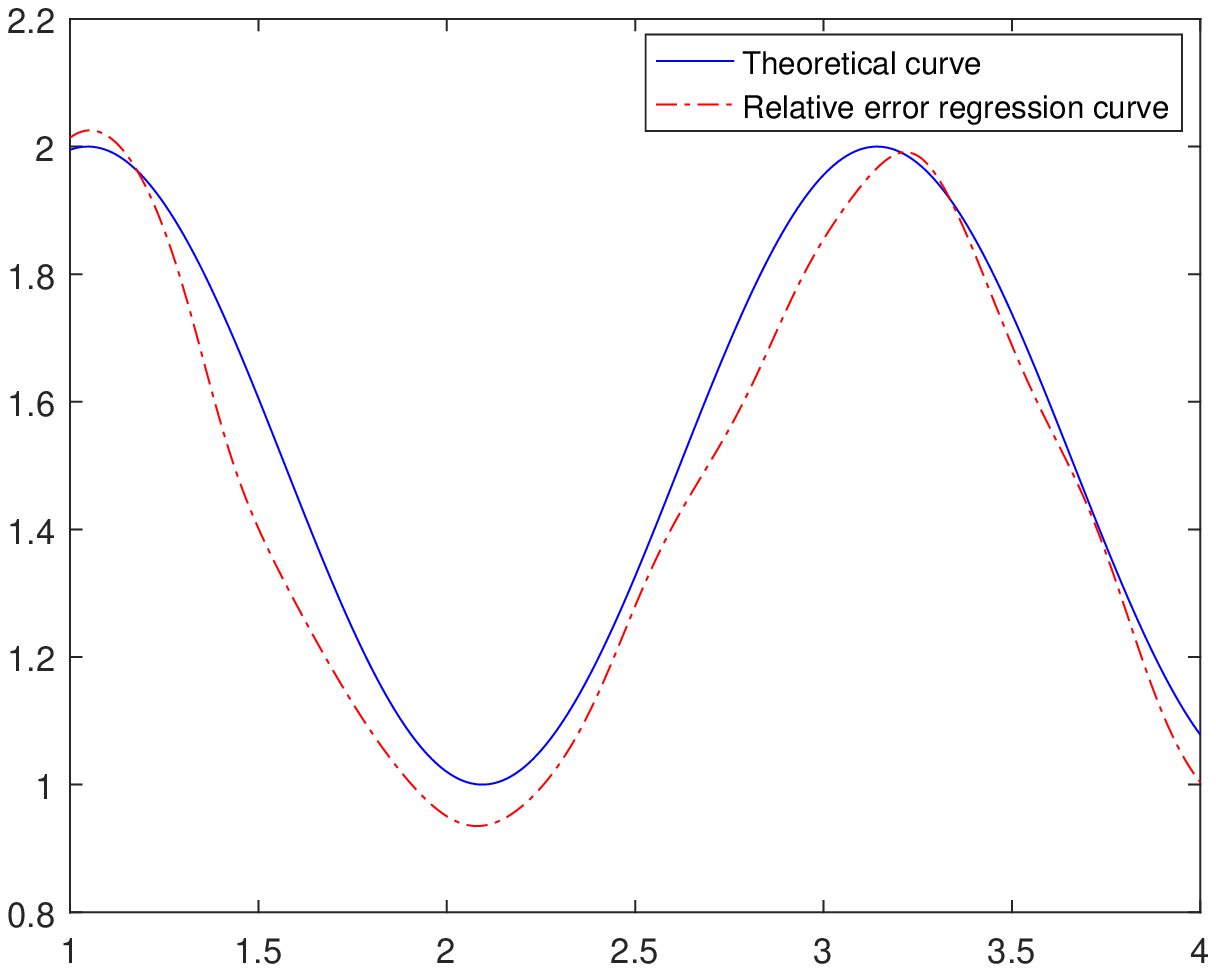}
	\end{minipage} \hfill
	\begin{minipage}[c]{.26\linewidth}
		\includegraphics[width=1.4\textwidth]{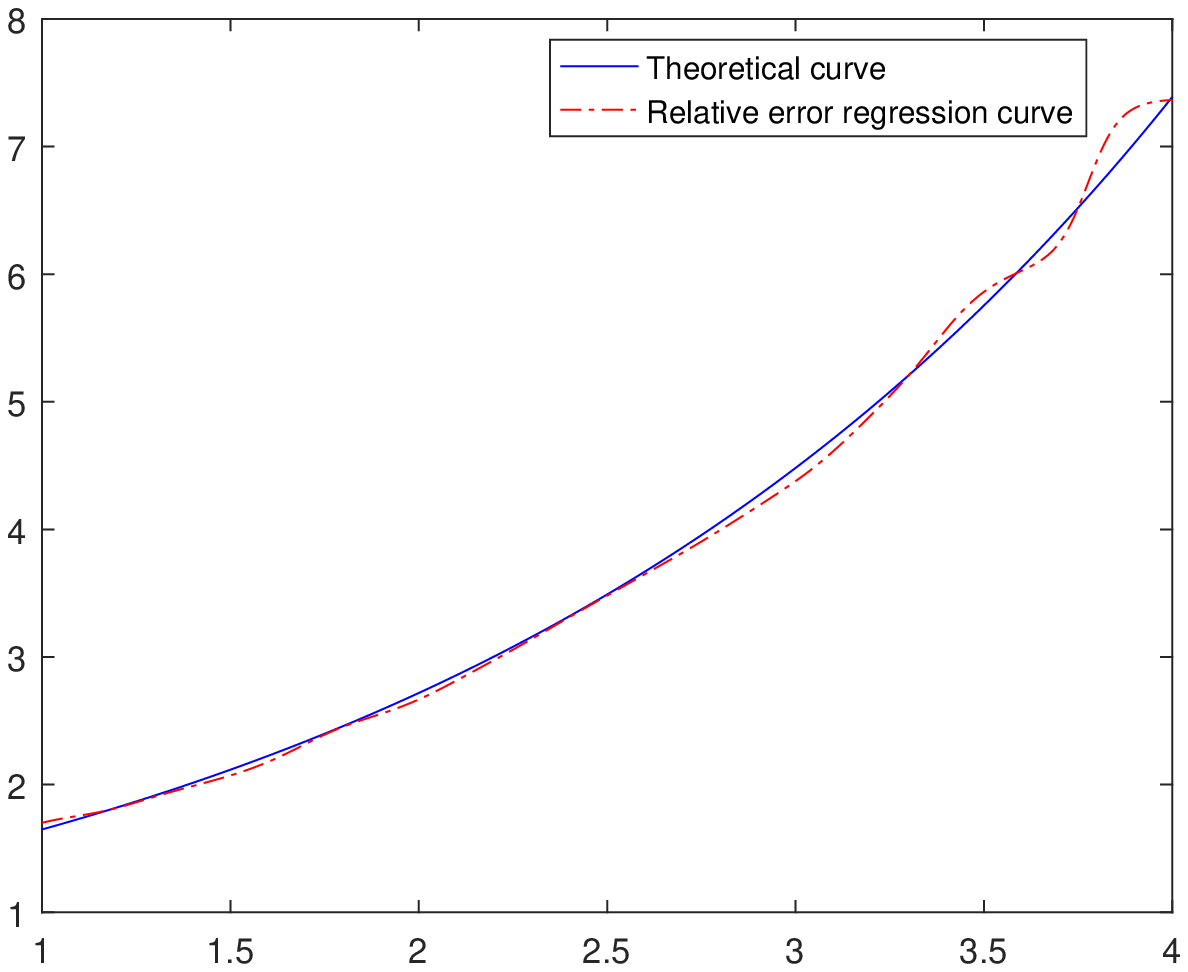}
	\end{minipage}\hfill\hfill
	\caption{$CR \approx 50\%$, $n=500$ for parabolic, sinus and exponential, respectively.}\label{figure3}
\end{figure}
\paragraph{Classical regression versus Relative error regression with respect to the censorship rate}
In order to highlight the efficiency of relative error estimation, we draw up a comparative study. For that, we simulated the classical regression estimator for randomly right censored data defined in \hyperref[guessoum]{Guessoum and Ould Sa\"\i d (2008)} by \[\displaystyle{\hat{r}(x)}=\frac{\displaystyle{\sum_{i=1}^{n}\frac{\delta_i Y_i}{\bar{G}_n(Y_i)}K\left(\frac{x-X_i}{h_n}\right)}}{\displaystyle{\sum_{i=1}^{n}K\left(\frac{x-X_i}{h_n}\right)}},\]
for the same parameters listed below. 
From Figure \ref{figure4} below, it is clear that the classical regression estimator deteriorates when the censorship rate increases considerably.
\begin{figure}[!h]
	\begin{minipage}[c]{.26\linewidth}
		\includegraphics[width=1.4\textwidth]{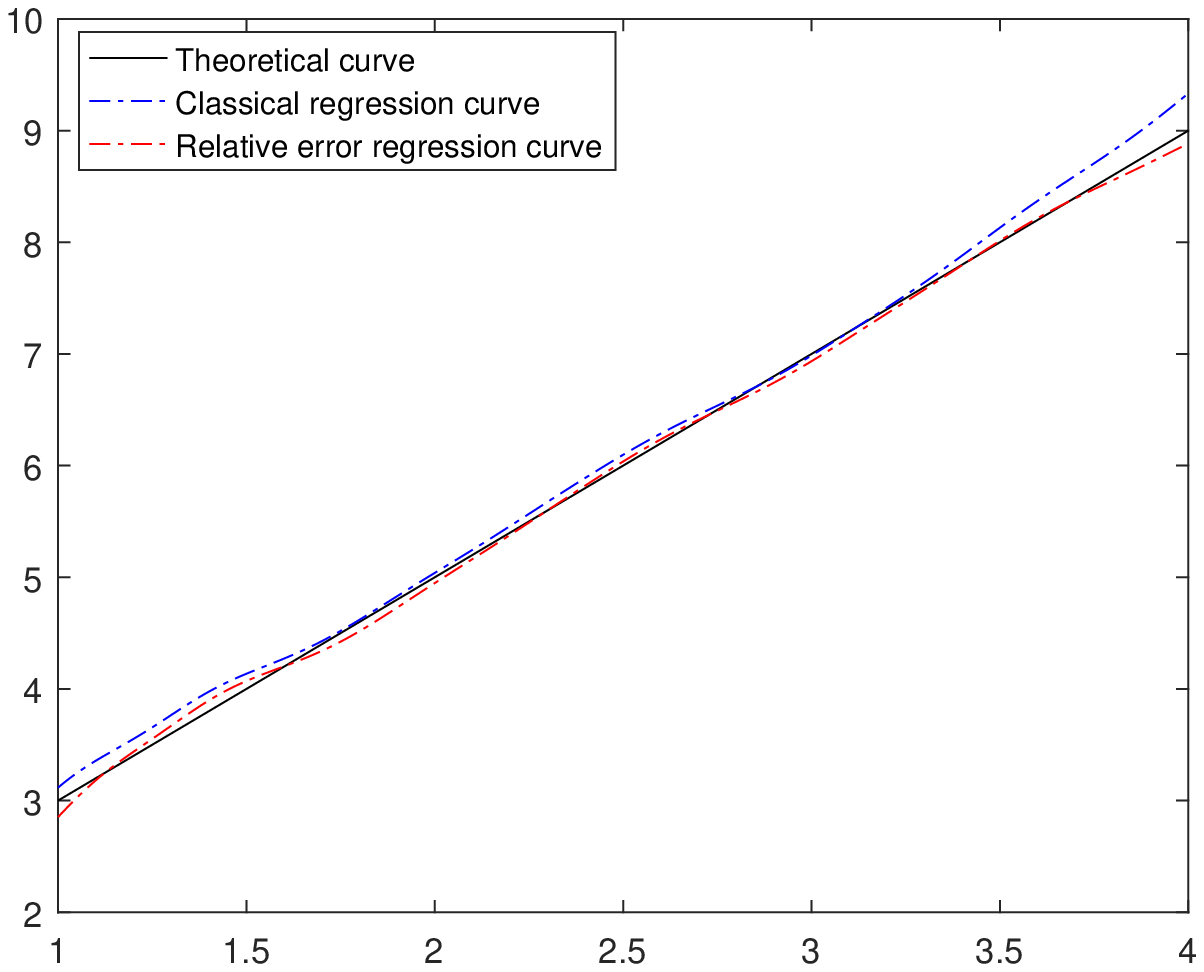}
	\end{minipage} \hfill
	\begin{minipage}[c]{.26\linewidth}
		\includegraphics[width=1.4\textwidth]{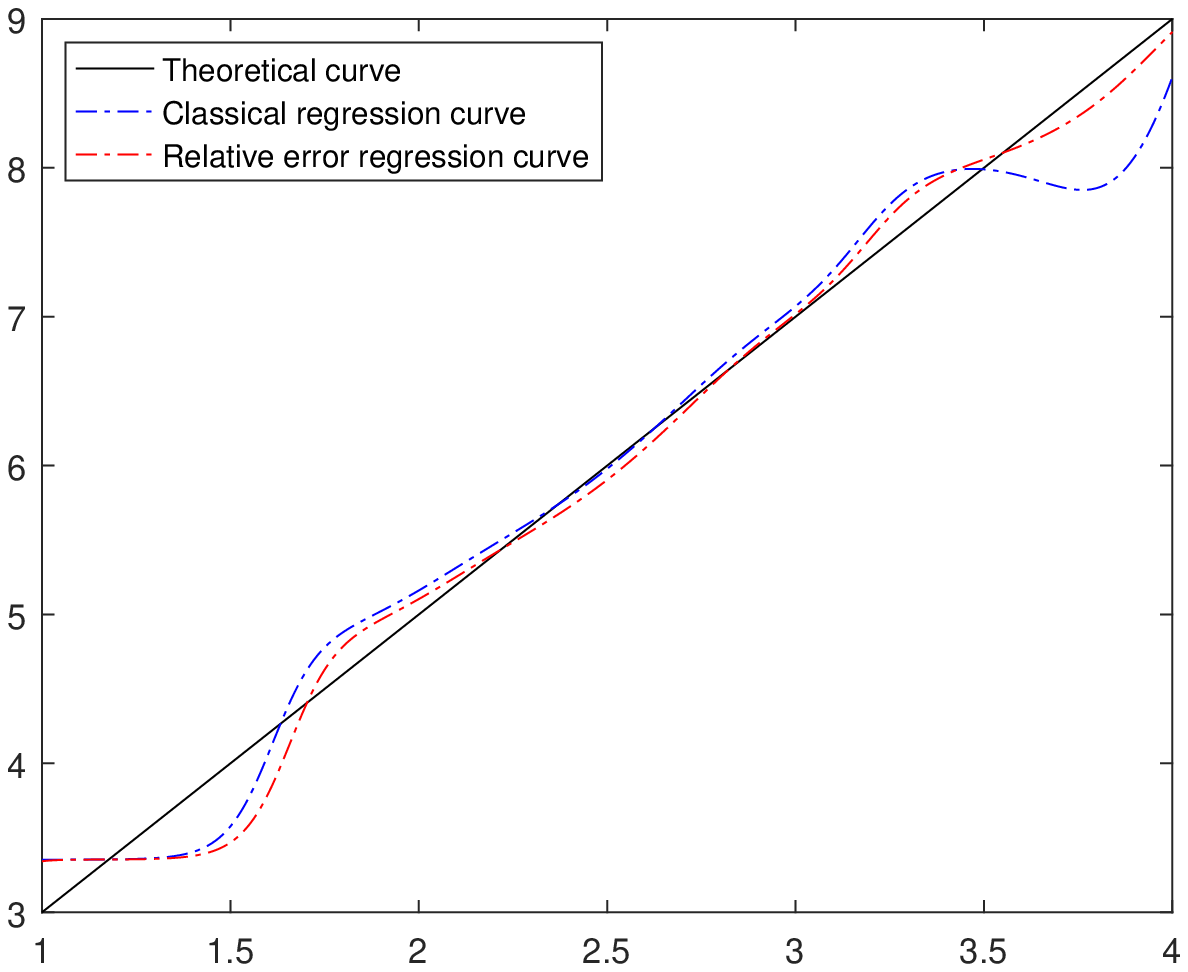}
	\end{minipage} \hfill
	\begin{minipage}[c]{.26\linewidth}
		\includegraphics[width=1.4\textwidth]{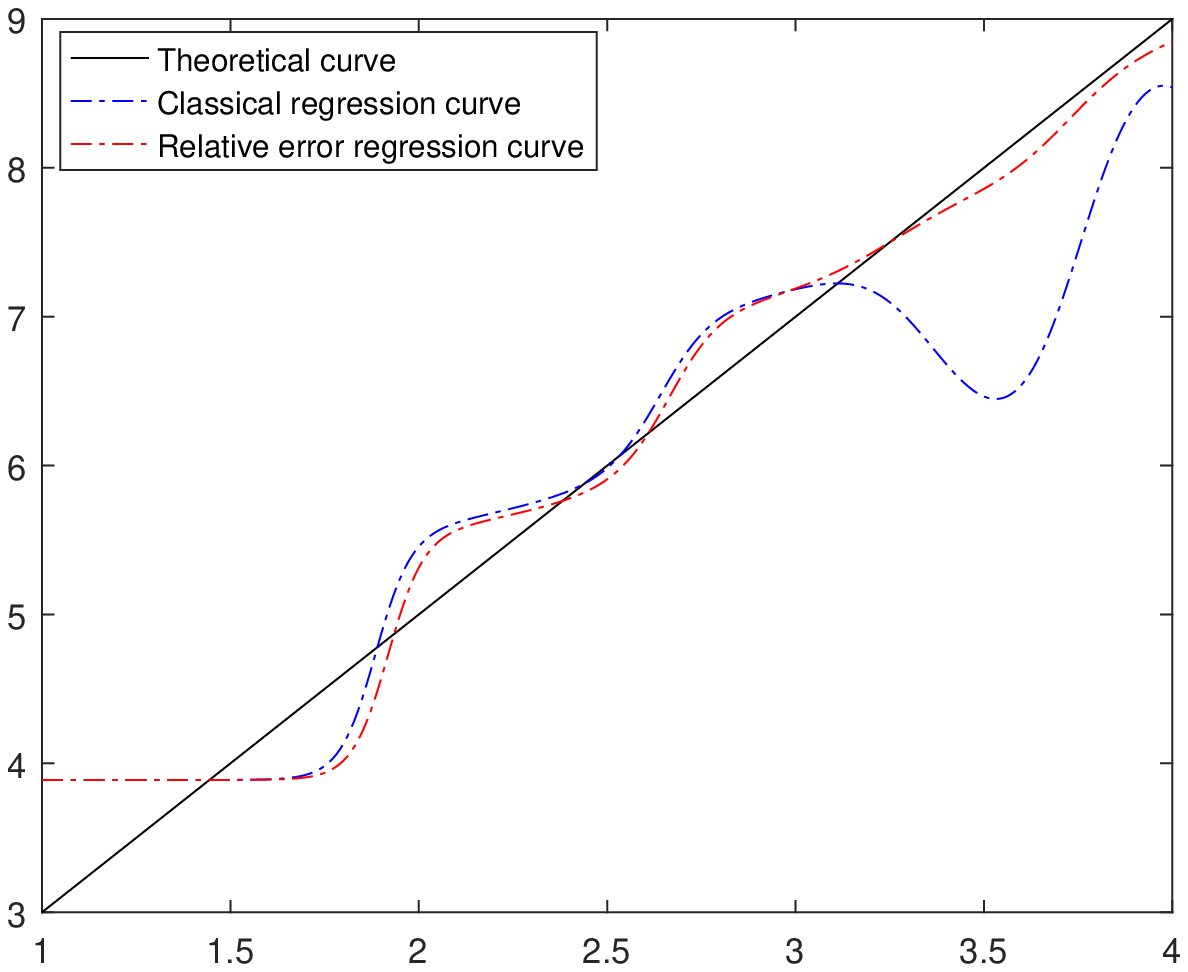}
	\end{minipage}\hfill\hfill
	\caption{$r(x)=2x+1$, $n=100$ and $CR\approx 15\%,50\%$ and $75\%$, respectively.}\label{figure4}
\end{figure}

\paragraph{Effect of outliers for the two methods with a fixed sample size and censorship rate.}
To show the robustness of our approach, we generate the case where the data contains outliers. For that we set both that sample size and censorship rate ($n=500$ and $TC \approx 50$). To create this outlier effect, 20 values of this sample are multiplied by a factor called $MF$. From Figure \ref{figure7}, we can see that our estimator is very close to the theoretical curve with respect to the classical one. Then, it is very clear that our approach is widely  better than the classical one in the presence of outliers.
\begin{figure}[!h]
	\begin{minipage}[c]{.26\linewidth}
		\includegraphics[width=1.4\textwidth]{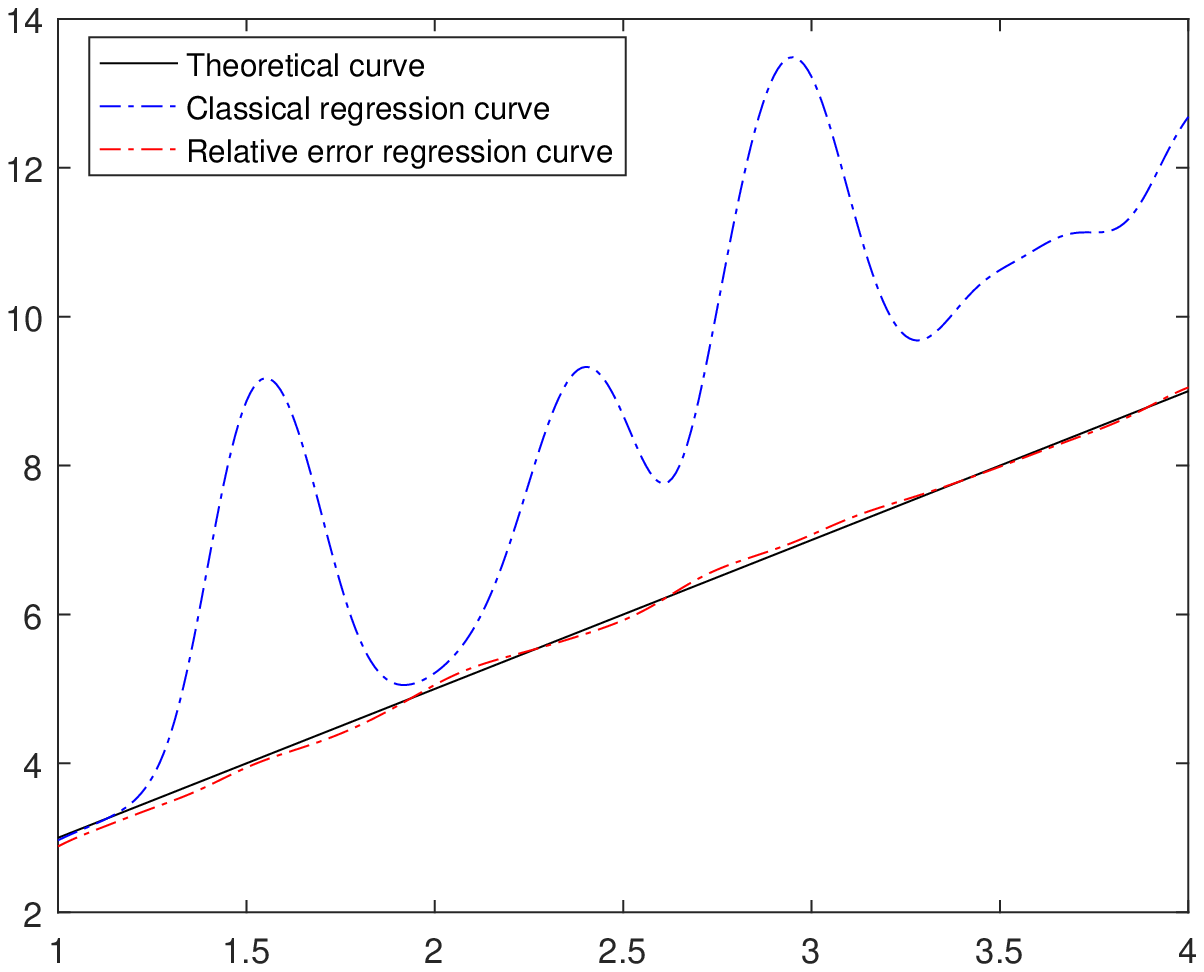}
	\end{minipage} \hfill
	\begin{minipage}[c]{.26\linewidth}
		\includegraphics[width=1.4\textwidth]{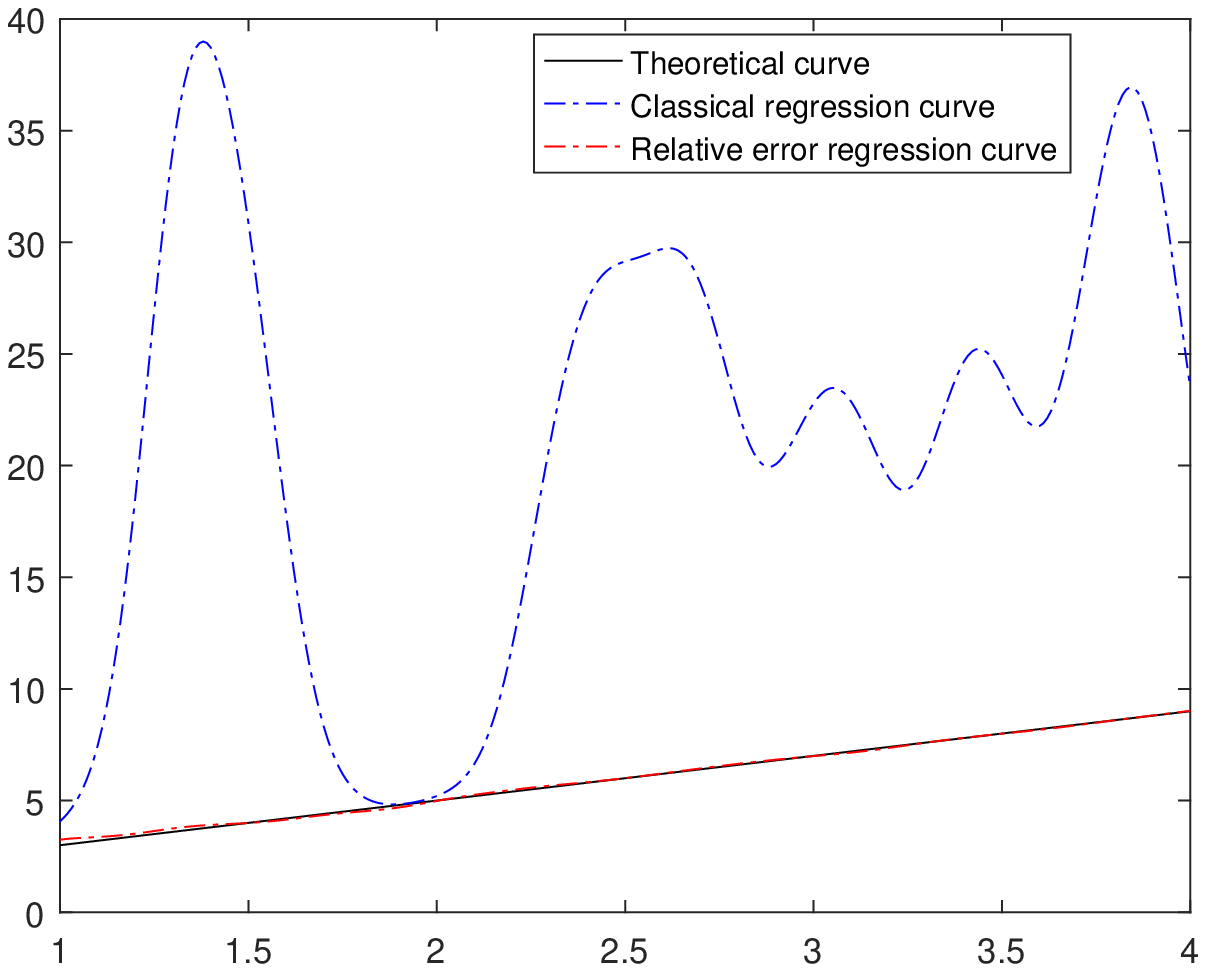}
	\end{minipage} \hfill
	\begin{minipage}[c]{.26\linewidth}
		\includegraphics[width=1.4\textwidth]{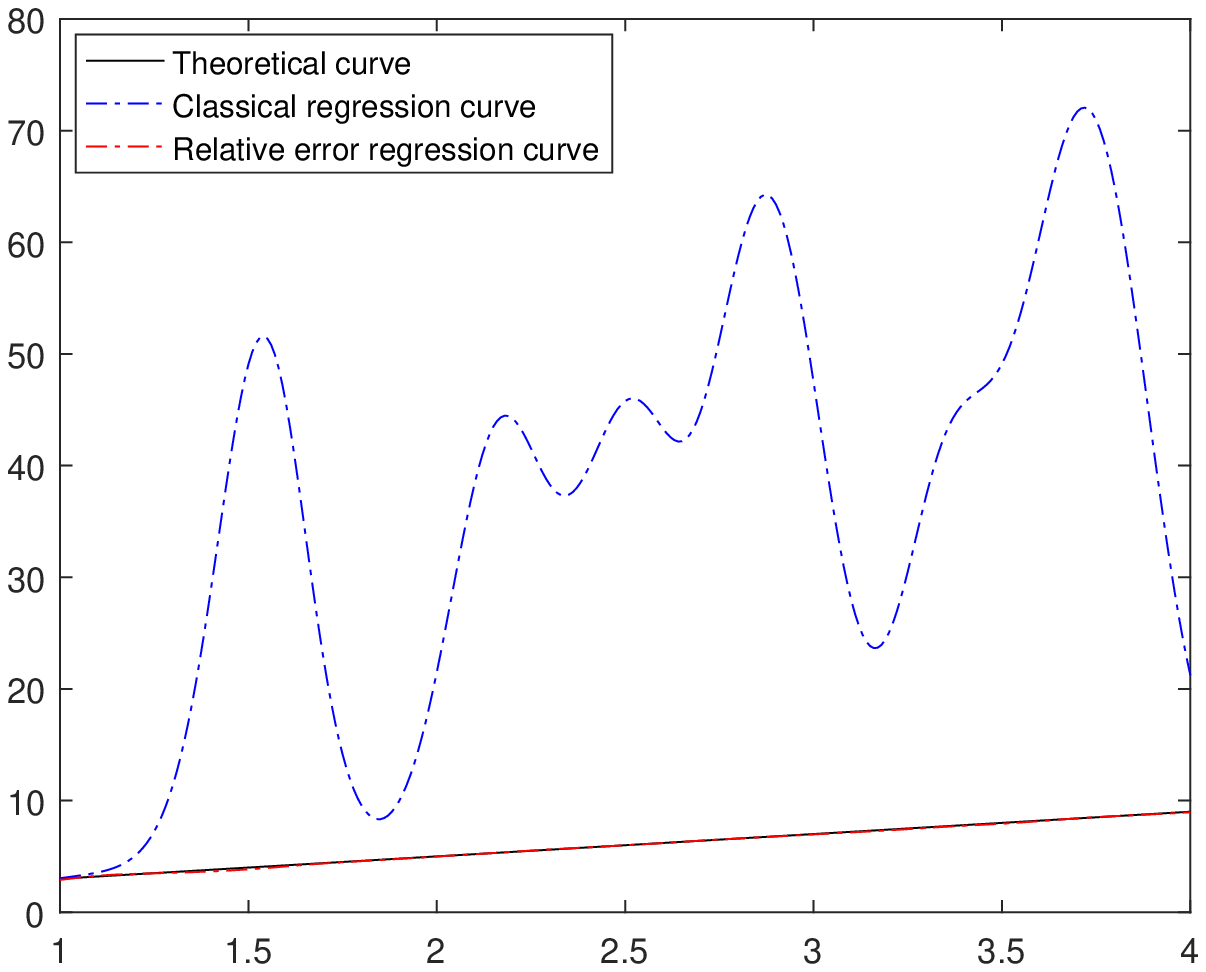}
	\end{minipage}\hfill\hfill
	\caption{ $r(x)=2x+1$, $n=500$, $ CR\approx 50\%$ and $ MF=10, 50$ and $100$ respectively.}\label{figure7}
\end{figure}
\newpage
\paragraph{Table of mean squared error.} 
To close this part, we compare the values of the mean squared error ($MSE$) for the classical regression ($MSE^1$) and the relative error regression ($MSE^2$) methods represented by
\[MSE^1=\frac{1}{n}\sum_{i=1}^{n}(T_i-\hat{r}(X_i))^2 \; \; \; \; \; \text{and} \; \; \; \; \; MSE^2=\frac{1}{n}\sum_{i=1}^{n}(T_i-r_n(X_i))^2\]
for three sample sizes and censoring level.\\
It can be seen from Table \ref{table1} that the variability of the mean squared error ($MSE$) of the two methods for a low censoring rate is not significantly considerable, i.e. the performance is the same for both methods. However, when the data is affected by the presence of censoring the $MSE$ of relative error regression becomes smaller than the classical regression. It means that the relative error regression model is more stable than the classical regression in the presence of censorship.
\begin{table}[!ht]
	\begin{tabular}{c|c|c|c}
		Sample & CR  & $MSE^1$ & $MSE^2$\\
		size &($\approx \%$)&(Classical Regression)&(Relative Error Regression)\\
		\hline \hline
		& 20 & 0.0150 & 0.0027 \\
		n=100	& 50 & 0.0802 & 0.0339\\
		& 80 & 0.4009 & 0.0885\\
		\hline
		& 20 & 0.0103 & 0.0057\\
		n=300	& 50 & 0.0275 & 0.0138\\
		& 80 & 0.1284  & 0.0425 \\
		\hline
		& 20 & 0.0078 & 0.0008 \\
		n=500 & 50 & 0.0138 & 0.0032\\
		& 80 & 0.1359 & 0.0136\\
		\hline
	\end{tabular}
	\centering
	\caption{The $MSE$ errors according to the censoring rates. } 
	\label{table1}
\end{table}

\subsection{Asymptotic normality}
The purpose of this part is to highlight the theoretical results obtained in Theorem \ref{theo2}, by studying by simulation the asymptotic normality. To do this, we compare the shape of the estimated density to that of the standard normal density in the case of a linear regression model:
\[T=2X+1+\epsilon,\]
we reproduce the same steps as in the previous subsection for $X \leadsto \exp(1.5)$ and $C \leadsto \exp(3)$. 
Throughout this subsection, we fix $x = 0$ and replicate $m$ independent $n$-sample size. Then, we calculate the asymptotic variance. For that we replace the $\Upsilon_k(.)$ for $k=2,3,4$ by their estimators in (\ref{upsilon2l}) and $\bar{r}_{\ell,n}(.)$ for $\ell=1,2$ by their estimators in (\ref{estimerell}). A calculable estimator of the normalized deviation is given by:
\[\sigma^2_n(0)=\kappa\frac{\hat{\Upsilon}_2(0)\bar{r}_{2,n}^2(0)-2\hat{\Upsilon}_3(0)\bar{r}_{1,n}(0)\bar{r}_{2,n}(0)+\hat{\Upsilon}_4(0)\bar{r}_{1,n}^2(0)}{\bar{r}_{2,n}^4(0)},\] 
we consider now the sequence:
\[A_j=\left(\frac{nh_n}{\sigma_{n,j}^2(0)}\right)^{1/2}(r_{n,j}(0)-1), \]
which under Theorem \ref{theo2}, $A_j$ follows asymptotically to $\mathcal{N}(0,1)$. Then, we build a kernel density estimator for the $A_k$ that we compare with the standard normal law for different values of $n$ and $h^*=c \left(\frac{\log(m)}{m}\right)^{0.2}$ where the constant $c$ is chosen appropriately. Finally, for a sample size $m=200$ and a censorship rate ($CR\approx 66\%$), we conduct $n=100, 300$ and $500$ replications. The figure \ref{figure6} show the quality of goodness of fit. 
\begin{figure}[!h]
	\begin{minipage}[c]{.26\linewidth}
		\includegraphics[width=1.4\textwidth]{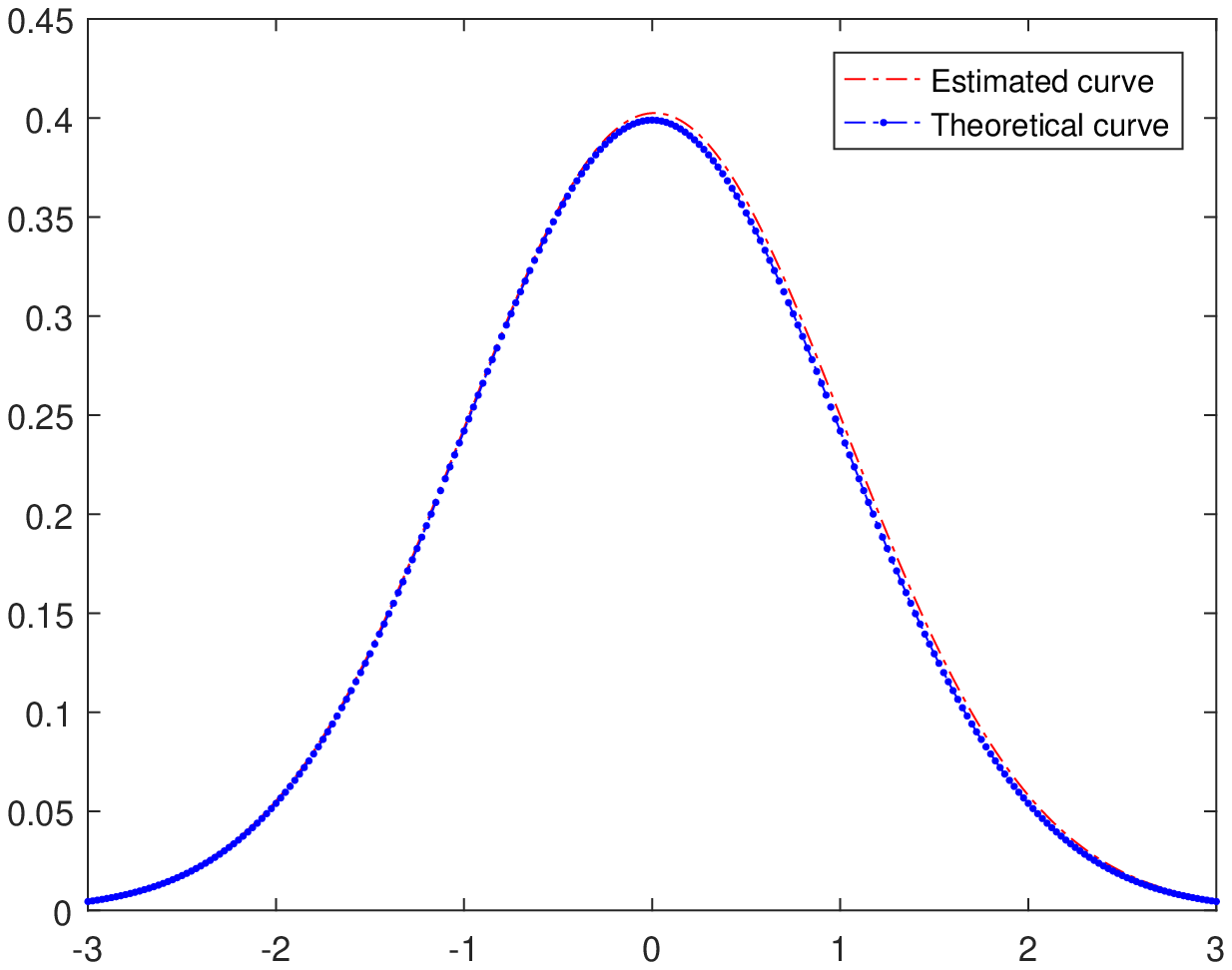}
	\end{minipage} \hfill
	\begin{minipage}[c]{.26\linewidth}
		\includegraphics[width=1.4\textwidth]{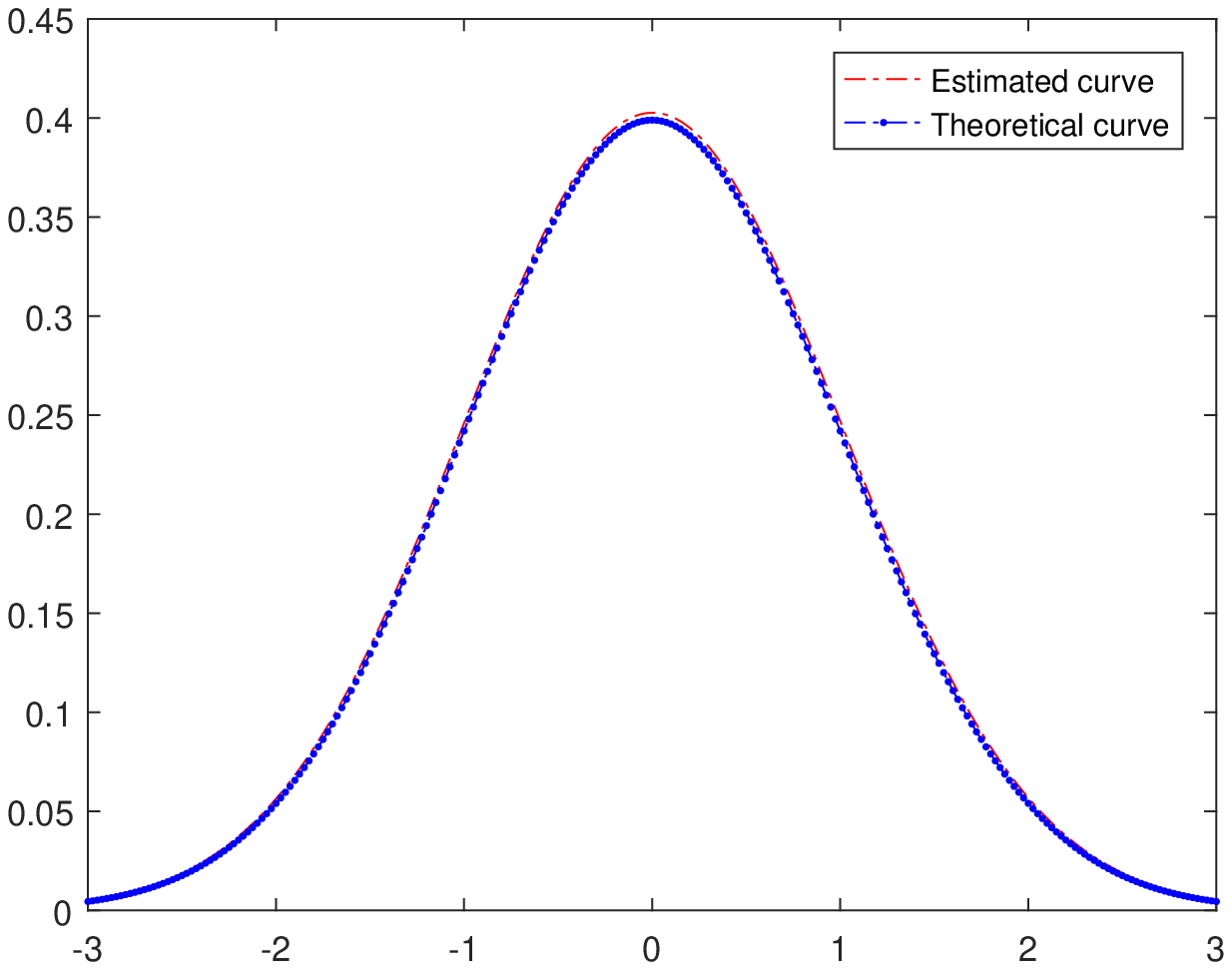}
	\end{minipage} \hfill
	\begin{minipage}[c]{.26\linewidth}
		\includegraphics[width=1.4\textwidth]{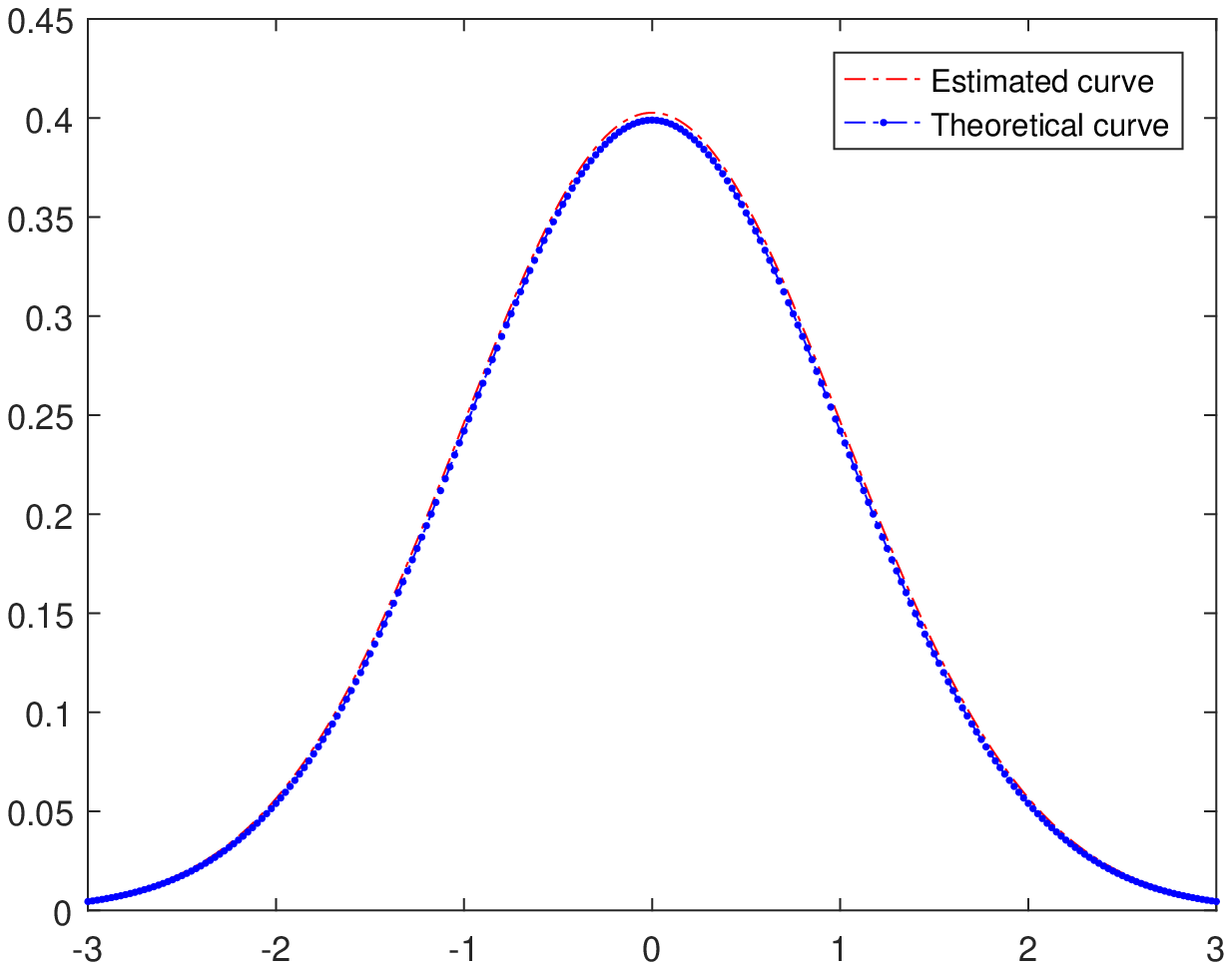}
	\end{minipage}\hfill\hfill
	\caption{ $CR\approx 66\%$ and $n=100,300$ and $500$, respectively.}\label{figure6}
\end{figure}
\section{Auxiliary results and proofs}\label{Sect 5}
\begin{proof} 
Using (\ref{P-S}), (\ref{rntilde}), (\ref{estimRRC}), (\ref{estimerell}) and for $x \in \mathbb{R}$,  we consider the following decomposition : 
		\begin{multline*}
	 r_{n}(x)-r(x) =\frac{1}{\bar{r}_{2,n}(x)} \left\{ \Big[\big(\bar{r}_{1,n}(x)-\bar{\tilde{r}}_{1,n}(x)\big) + ( \bar{\tilde{r}}_{1,n}(x)-\mathbb{E}[\bar{\tilde{r}}_{1,n}(x)])+ (\mathbb{E}[\bar{\tilde{r}}_{1,n}(x)]-\bar{r}_{1}(x))\Big] \right.\\ 
	+ r(x) \left. \Big[( \bar{\tilde{r}}_{2,n}(x)-\bar{r}_{2,n}(x))+( \mathbb{E}[\bar{\tilde{r}}_{2,n}(x)]-\bar{\tilde{r}}_{2,n}(x))+  (\bar{r}_{2}(x)-\mathbb{E}[\bar{\tilde{r}}_{2,n}(x)])\Big] \right\}
	\end{multline*}
which by triangle inequality, we have
	\begin{multline*}
	\sup_{x \in \mathcal{C}} | r_{n}(x)-r(x) |\\
	\leq \frac{1}{\inf_{x \in \mathcal{C}}|\bar{r}_{2,n}(x)|}\{\sup_{x \in \mathcal{C}} \{|\bar{r}_{1,n}(x)-\bar{\tilde{r}}_{1,n}(x)| + | \bar{\tilde{r}}_{1,n}(x)-\mathbb{E}[\bar{\tilde{r}}_{1,n}(x)]|+ | \mathbb{E}[\bar{\tilde{r}}_{1,n}(x)]-\bar{r}_{1}(x)|\} \\ 
	+ \sup_{x \in \mathcal{C}}|r(x)|\{| \bar{\tilde{r}}_{2,n}(x)-\bar{r}_{2,n}(x)|+| \mathbb{E}[\bar{\tilde{r}}_{2,n}(x)]-\bar{\tilde{r}}_{2,n}(x)|+ | \bar{r}_{2}(x)-\mathbb{E}[\bar{\tilde{r}}_{2,n}(x)]|\}\}.
	\end{multline*}
	In the sequel, we give a sequence of lemmas that are helpful in proving our results.
	\begin{lem}\label{lem1}
		Under hypotheses \hyperref[h]{\bf{\textit{H}}}\textit{i)} and \hyperref[k]{\bf{\textit{K}}}, we have, for $ \ell=1,\,2$, that
		\begin{equation}
		\sup_{x\in \mathcal{C}} | \bar{r}_{\ell,n}(x)-\bar{\tilde{r}}_{\ell,n}(x)| = O_{a.s.}\left\{\left(\frac{\log_{2}n}{n}\right)^{1/2}\right\} \quad as \quad n \longrightarrow \infty. 
		\end{equation}
	\end{lem}
	\begin{proof}
		For $\ell=1,\,2$, we have
		\[\begin{aligned}
		| \bar{r}_{\ell,n}(x)-\bar{\tilde{r}}_{\ell,n}(x)| & =\frac{1} {n h_n} \left|\sum_{i=1}^n \frac{ \delta_i Y_i^{-\ell}}{\overline{G}_n(Y_i)}K\left(\frac{x-X_i}{h_n}\right)-\frac{ \delta_i Y_i^{-\ell}}{\overline{G}(Y_i)}K\left(\frac{x-X_i}{h_n}\right)\right|\\
		& = \frac{1}{n h_n}\left|\sum_{i=1}^n\frac{T_i^{-\ell}}{\overline{G}_n(T_i)}\mathds{1}_{\{T_{i}\leq C_{i}\}}K\left(\frac{x-X_i}{h_n}\right)- \frac{T_i^{-\ell}}{\overline{G}_n(T_i)}\mathds{1}_{\{T_{i}\leq C_{i}\}}K\left(\frac{x-X_i}{h_n}\right) \right| \\
		& \leq \frac{1}{n h_n} \sum_{i=1}^{n} \left| {T_{i}^{-\ell} K\left(\frac{x-X_i}{h_n}\right)} \right| \left| \frac{1}{\overline{G}_{n}(T_i)}-\frac{1}{\overline{G}(T_i)} \right| \\
		& \leq \frac{1}{n h_n} \sum_{i=1}^{n} \left| {T_{i}^{-\ell} K\left(\frac{x-X_i}{h_n}\right)} \right| \left| \frac{\overline{G}_{n}(T_i)-\overline{G}(T_i)}{\overline{G}_{n}(T_i) \overline{G}(T_i)} \right| \\
		& \leq \frac{\sup_{t< \tau_F} \mid \overline{G}_{n}(t)-\overline{G}(t) \mid }{\overline{G}_{n}(\tau_F) \overline{G}(\tau_F)}  \frac{M}{n h_n} \sum_{i=1}^{n} \left| K \left(\frac{x-X_i}{h_n}\right) \right|,
		\end{aligned}\]
		then by using the strong law of large numbers (SLLN) and law of iterated logarithm (LIL) on the censoring law (see formula (4.28) in \hyperref[deheuvels]{Deheuvels and Einmahl, 2000}), we get,
		\[\sup_{x\in \mathcal{C}} \mid \bar{r}_{\ell,n}(x)-\bar{\tilde{r}}_{\ell,n}(x)\mid \leq \frac{M}{\overline{G}^{2}(\tau_F)} \mathbb{E} \left(\left| \frac{1}{h_n}K\left(\frac{x-X_1}{h_n}\right)\right| 
		\right) \sqrt {\frac{\log_{2}n}{n}}.\]
		Then hypotheses \hyperref[h]{\bf{\textit{H}}}\textit{i)} and \hyperref[k]{\bf{\textit{K}}} complete the proof of the lemma.
	\end{proof}
	\begin{lem}\label{lem2}
		Under hypotheses \hyperref[h]{\bf{\textit{H}}}\textit{i)}, \hyperref[k]{\bf{\textit{K}}} and \hyperref[d]{\bf{\textit{D}}}\textit{i)}, we have, for $ \ell =1,2$, that:
		\begin{equation}
		\sup_{x\in \mathcal{C}} \mid \mathbb{E}[\bar{\tilde{r}}_{\ell,n}(x)]-\bar{r}_{\ell}(x) \mid = {O}(h_{n}^{2})  \quad as\qquad n \longrightarrow \infty.
		\end{equation}
	\end{lem}
	\begin{proof}
		Using the conditional expectation properties, we get,
		\[\begin{aligned}
		\mathbb{E}\left[\bar{\tilde{r}}_{\ell,n}(x)\right]& = \frac{1}{h_n} \mathbb{E} \left[ \frac {\delta_{1} Y_{1}^{-\ell}}{\overline{G}(Y_{1})} {K\left(\frac{x-X_{1}}{h_n}\right)}\right]\\
		&=\frac{1}{h_n} \mathbb{E}\left[{K\left(\frac{x-X_{1}}{h_n}\right)}\mathbb{E}\left(\frac {\delta_{1} Y_{1}^{-\ell}}{\overline{G}(Y_{1})}|X_{1}\right)\right]\\
		&= \frac{1}{h_n} \int{K\left(\frac{x-u}{h_n}\right)}\mathbb{E}\left[\frac {\delta_{1} Y_{1}^{-\ell}}{\overline{G}(Y_{1})}|X_{1}=u\right]f(u)du,
		\end{aligned}\]
		and as ${\mathds{1}_{\{T_{1} \leq C_{1} \}}Y_{1}^{-\ell}}={\mathds{1}_{\{T_{1} \leq C_{1} \}}T_{1}^{-\ell}}$ , we get
		\begin{equation}
		\mathbb{E}\left[\frac {T^{-\ell}}{\overline{G}(T_{1})}\mathbb{E}\left[\mathds{1}_{\{T_{1} \leq C_{1} \}}|T_{1}\right]|X_{1}=u\right]=r_{\ell}(u).
		\end{equation}
		By a change of variable and using $\bar{r}_{\ell}(\cdot)=r_{\ell}(\cdot)f(\cdot)$, we get  
		\[\begin{aligned}
		\left| \mathbb{E} \left( \frac{1}{h_n} \frac {\delta_{1}Y^{-\ell}}{\overline{G}(Y)} K\left(\frac{x-X}{h_n}\right)\right)-{\bar{r}_{\ell}(x)}\right| &= \left| \int \frac{1}{h_n}{K\left(\frac{x-u}{h_n}\right)}r_{\ell}(u)f(u)du-{\bar{r}_{\ell}(x)}\right|\\
		&=\left| \int \frac{1}{h_n}{K\left(\frac{x-u}{h_n}\right)}\bar{r}_{\ell}(u)du-{\bar{r}_{\ell}(x)}\right|\\
		&=\left| \int K(t)[\bar{r}_{\ell}(x-h_{n}t)-{\bar{r}_{\ell}(x)}]dt\right|.
		\end{aligned}\]
		We use a Taylor expansion to $\bar{r}_\ell (\cdot)$ for $\zeta \in ]x-h_{n}t,x[$, we get
		\[\begin{aligned}
		\sup_{x\in \mathcal{C}} \left|(\mathbb{E}(\bar{\tilde{r}}_{\ell,n}(x))-\bar{r}_{\ell}(x) \right| & = \sup_{x\in \mathcal{C}} \left| \int K(t)[-h_{n}t \bar{r}_{\ell}^{'}(x)-\frac{h_{n}^{2}t^2}{2} \bar{r}^{''}_{\ell}(\zeta)]dt\right| \\ 
		& \leq h_n \sup_{x\in \mathcal{C}} \left| \int t K(t) \bar{r}_{\ell}^{'}(x) dt \right| + h_{n}^{2} \sup_{x\in \mathcal{C}} \left| \int \frac {t^2}{2} K(t) \bar{r}^{''}_{\ell}(\zeta)dt\right|\\
		& \leq h_n \sup_{x\in \mathcal{C}} \left| \int t K(t) \bar{r}_{\ell}^{'}(x) dt\right| + \frac{h_{n}^{2}}{2} \sup_{x\in \mathcal{C}}\left| \int \bar{r}_{\ell}^{''}(\zeta) t^2 K(t)dt\right|.
		\end{aligned}\]
		Under Hypotheses \hyperref[h]{\bf{\textit{H}}}\textit{i)} and \hyperref[k]{\bf{\textit{K}}}\textit{ii)}, the first term is equal to zero. The second term goes to zero for $n \rightarrow \infty$ from hypotheses \hyperref[d]{\bf{\textit{D}}}\textit{i)} and  \hyperref[k]{\bf{\textit{K}}}\textit{ii)}. The last result complete the proof of the lemma. 
	\end{proof}
	\begin{lem}\label{lem3}
		Under hypotheses \hyperref[h]{\bf{\textit{H}}}\textit{i)} and \hyperref[k]{\bf{\textit{K}}}\textit{i)}, we have, for $ \ell=1,2 $ 
		\begin{equation}
		\sup_{x\in \mathcal{C}} \mid \bar{\tilde{r}}_{\ell,n}(x)-\mathbb{E}[\bar{\tilde{r}}_{\ell,n}(x)]\mid =O_{a.s.}\left(\sqrt{\frac{\log n}{nh_{n}}}\right)  \quad as \quad n \longrightarrow \infty.
		\end{equation}
	\end{lem}
	\begin{proof}
		Let us consider the i.i.d sequence $(X_{1},Y_{1},\delta_{1}),\dots,(X_{n},Y_{n},\delta_{n})$ and define 
		\[\Phi_{n}= \left\{ \theta_{x} :  \mathbb{R} \times \mathbb{R}_+^* \times \{0,1\}\rightarrow \mathbb{R}^+ / \theta_{x}(u,y,\delta)=\frac{\delta y^{-\ell}}{n h_n \overline{G}(y)} K{\left(\frac{x-u}{h_n}\right)}, \quad x \in \mathbb{R} \right\}.\]
		By Lemma (3b) in \hyperref[gine1]{Gin\'{e} and Guillou (1999)}, $\Phi_n$ is Vapnik-Cervonenkis (V-C) class of no-negative measurable functions. These are uniformly bounded with respective envelopes $\displaystyle \Theta = \frac{M\norm{K}_{\infty}}{n h_n\, 
		\overline{G}(\tau_F)}$. Moreover,
		\[\begin{aligned}
		\sup_{x \in \mathcal{C}}\mathbb{E}\left[\theta_{x}(X_{1},Y_{1},\delta_{1})\right] &\leq \frac{M \norm{K}_{\infty}}{n h_n \overline{G}(\tau_F) }=: U_n
		\end{aligned}\]		
		In the same way, we get
		\[\begin{aligned}
		\sup_{x \in \mathcal{C}}Var\left[\theta_{x}(X_{1},Y_{1},\delta_{1})\right] & \leq \sup_{x \in \mathcal{C}} \mathbb{E}\left[\theta_{x}^{2}(X_{1},Y_{1},\delta_{1})\right] \\
		& \leq \frac{M^2 \norm{K}^{2}_{2} \norm {f}_{\infty}}{n^2 h_{n} \overline{G}^2(\tau_F)} =: \sigma_{n}^{2}
 		\end{aligned}\]
		with $\sigma_n \leq U_{n}$ for $n$ large enough. \\
		Now applying Talagrand's inequality [see Proposition 2.2 in \hyperref[gine]{Gin\'{e} and Guillou (2001)}], with $t \geq A \sqrt{\frac {\log n}{n h_n}}$, there exist two positives constants $L$ and $B$ such that 
		\[\mathbb{P} \left[ \sup_{\theta_{x} \in \Phi_{n}} \left|  \sum_{i=1}^{n}\left(\theta_{x}(X_{i},Y_{i},\delta_{i}) -\mathbb{E}[ \theta_{x}(X_{1},Y_{1},\delta_{1})] \right) \right| > A \sqrt{\frac {\log n}{nh_{n}}}\right]\]
		\[\leq L \exp \left( -\frac{A \sqrt {\frac{\log n}{nh_{n}}}}{L \frac{M \norm{K}_{\infty}}{n h_n \bar{G}(\tau_F)}}\log \left[1+\frac{A \frac {{\sqrt {\log n}} \norm{K}_{\infty}}{(nh_n)^{3/2} \bar{G}(\tau_F)}}{L \left(\frac{ M\norm{K}_{2} \sqrt{\norm{f}_{\infty}}}{\sqrt{nh_n} \overline{G}(\tau_F)}+ \frac{M \norm{K}_{\infty}}{n h_n \bar{G}(\tau_F)} \sqrt {\log{B \frac{\norm{K}_{\infty}}{\sqrt{h_n} \norm{K}_{2}^{2} \sqrt{\norm{f}_{\infty}}}}}\right)^{2} }\right]\right),\]
		and using $\log(1+x) \approx x  $( for $ x \rightarrow 0 )$, the right-hand of the last equation becomes an order of 
		\[L \exp \left( -\frac{A \bar{G}(\tau_F) \sqrt {\frac{\log n}{nh_{n}}}}{L {M \norm{K}_{\infty}}} n h_n \frac{A \frac {{\sqrt {\log n}} \norm{K}_{\infty}}{(nh_n)^{3/2}\bar{G}(\tau_F)}}{L \left(\frac{ M \norm{K}_{2} \sqrt{\norm{f}_{\infty}}}{\sqrt{nh_n} \overline{G}(\tau_F)}\right)^2} \right)= L n^{- \frac{ \overline{G}^2(\tau_F)}{ M^2 \norm{K}_{2}^{2} \norm{f}_{\infty} } \left(\frac{A}{L}\right)^2},\]
		which by an appropriate choice of the constant $A$, can be made $O(n^{-3/2})$. The latter being a general term of summable series and by Borel-Cantelli's lemma we conclude the proof.
	\end{proof}
\noindent 	Then \hyperref[lem1]{Lemma 1}- \hyperref[lem3]{Lemma 3} permit to conclude the proof of the Theorem \ref{theo1}.
\end{proof}
\noindent Next we proceed to the proof of the Theorem \ref{theo2}.
\begin{proof}
	Our goal is to show 
		\[\sqrt{nh_n}(r_{n}(x)-r(x))\xrightarrow{\  \mathcal{D}    \ } \mathcal{N}(0,\sigma^2(x)) \quad \text{as} \quad n \longrightarrow \infty.\]
	Note that, for $\ell=1,2$,
	\[\begin{aligned}
	\sqrt{nh_n}(\bar{r}_{\ell,n}(x)-\bar{r}_{\ell}(x)) &= \sqrt{nh_n}\left(\bar{r}_{\ell,n}(x)-\bar{\tilde{r}}_{\ell,n}(x)\right)+\sqrt{nh_n}\left(\bar{\tilde{r}}_{\ell,n}(x)-\mathbb{E}[\bar{\tilde{r}}_{\ell,n}(x)]\right)\\
	& +\sqrt{nh_n}\left(\mathbb{E}[\bar{\tilde{r}}_{\ell,n}(x)]-\bar{r}_{\ell}(x)\right)\\
	& =:\Lambda_{\ell,n}(x)+\Gamma_{\ell,n}(x)+\Sigma_{\ell,n}(x).
	\end{aligned}\]
First, we consider the negligible terms $\Lambda_{\ell,n}$ and $\Sigma_{\ell,n}$.
\begin{lem}
Under \hyperref[h]{\bf{\textit{H}}}\textit{ii)},\textit{iii)} and by \hyperref[lem1]{Lemma 1}, \hyperref[lem2]{Lemma 2},  both  $\sqrt{nh_n}\Lambda_{\ell,n}(x)$ and  $\sqrt{nh_n}\Sigma_{\ell,n}(x)$ are $o(1)$ as $n \rightarrow \infty$.
\end{lem}
\begin{proof}
\noindent  From \hyperref[lem1]{Lemma 1}, under \hyperref[h]{\bf{\textit{H}}}\textit{ii)}, we get,
\begin{equation}\label{gamma}
\Lambda_{\ell,n}(x)=\sqrt{nh_n}\left(\bar{r}_{\ell,n}(x)-\bar{\tilde{r}}_{\ell,n}(x)\right)=O_{a.s.}\left(\sqrt{h_{n}\log_{2}n}\right) = o_{a.s.}(1).
\end{equation}
In the same way, from \hyperref[lem2]{Lemma 2}, under \hyperref[h]{\bf{\textit{H}}}\textit{iii)}, we have,
\begin{equation}\label{sigma}
\Sigma_{\ell,n}(x) =\sqrt{n h_n}\left(\mathbb{E}[\bar{\tilde{r}}_{\ell,n}(x)]-\bar{r}_{\ell}(x)\right)=O\left(\sqrt{nh_{n}^{5}}\right)= o(1).
\end{equation}	
\end{proof}
Now we consider the dominant terms $\Gamma_{\ell,n}(x)$ for $\ell \in \{1,2\}$ and prove Lemma \ref{var_cov}.
\begin{lem}\label{var_cov}
Under hypotheses \hyperref[h]{\bf{\textit{H}}}\textit{i)},\hyperref[k]{\bf{\textit{K}}} and \hyperref[d]{\bf{\textit{D}}}\textit{i)}, \textit{ii)}, we have
\begin{equation*}
(\Gamma_{1,n}(x),\Gamma_{2,n}(x))^T\xrightarrow{\  \mathcal{D}    \ }\mathcal{N}(0,\kappa\Sigma(x)) \quad \quad \text{as} \quad n \rightarrow \infty.
\end{equation*}
\end{lem}
\begin{proof}
	We first estimate the asymptotic variance, for $\ell=1,2$, we get
	\begin{equation*}
	\begin{aligned}
	Var\left[\Gamma_{\ell,n}(x)\right]&=n h_n \left\{ \mathbb{E} \left[ \bar{\tilde{r}}_{\ell,n}^{2}(x) \right] - \mathbb{E}^{2} \left[\bar{\tilde{r}}_{\ell,n}(x)\right]\right\}\\ 
	&=h_{n}^{-1} \left\{ \mathbb{E} \left[ \frac{\delta_1 Y_1^{-2\ell}}{\overline{G}^{2}(Y_1)} K^{2}\left(\frac{x-X_1}{h_n}\right)  \right] - \mathbb{E}^{2} \left[ \frac{\delta_1 Y_1^{-\ell}}{\overline{G}(Y_1)} K\left(\frac{x-X_1}{h_n}\right)\right] \right\}\\
	&=\mathcal{V}_{1}-\mathcal{V}_{2}.
	\end{aligned}
	\end{equation*}
	For $\mathcal{V}_{2}$ proceeding as in \hyperref[lem2]{Lemma 2} and under \hyperref[h]{\bf{\textit{H}}}\textit{i)},\hyperref[k]{\bf{\textit{K}}} and \hyperref[d]{\bf{\textit{D}}}\textit{i)}, we have
	\begin{equation}\label{V1}
	\mathcal{V}_{2}= h_{n}^{-1} \left[h_{n}\int K(s)\bar{r}_{\ell}(x-h_{n}s)ds\right]^{2} = o(1).
	\end{equation}
	Furthermore for $\mathcal{V}_{1}$, we get, 
	\begin{equation*}
	\begin{aligned}
	\mathcal{V}_{1}= h_{n}^{-1} \mathbb{E} \left[ \frac{\delta_1 Y_1^{-2\ell}}{\overline{G}^{2}(Y_1)} K^{2}\left(\frac{x-X_1}{h_n}\right) \right]& = h_{n}^{-1} \mathbb{E} \left[ K^{2}\left(\frac{x-X_1}{h_n}\right)\mathbb{E} \left[ \frac{ T_1^{-2\ell}}{\overline{G}(T_1)} |X_1 \right] \right]\\
	&= h_{n}^{-1} \int K^{2}\left(\frac{x-u}{h_n}\right) \int \frac{t^{-2\ell}}{\overline{G}(t)} f_{T_1|X_1}(t|u)dt f(u)du\\
	&= h_{n}^{-1} \int K^{2}\left(\frac{x-u}{h_n}\right) \int \frac{t^{-2\ell}}{\overline{G}(t)} f_{T_1,X_1}(t,u)dtdu\\
	&= h_{n}^{-1} \int K^{2}\left(\frac{x-u}{h_n}\right) \Upsilon_{2\ell}(u)du
	\end{aligned}
	\end{equation*}
	by a change of variable and Taylor expansion, by \hyperref[k]{\bf{\textit{K}}}\textit{iii)}, we get, 
	\begin{equation}\label{V2}
	\begin{aligned}
	&= h_{n}^{-1} \int K^{2}(t) \Upsilon_{2\ell}(x-h_{n}t) h_{n}dt\\
	& = \Upsilon_{2\ell}(x) \int K^{2}(t) dt + o(h_n)\\ 
	& = \Upsilon_{2\ell}(x)\kappa+o(h_n)
	\end{aligned}
	\end{equation}
which together with (\ref{V1}) and (\ref{V2}) gives under \hyperref[d]{\bf{\textit{D}}}\textit{ii)}, for $\ell=1,2$ 
\begin{equation*}
Var\left[\Gamma_{\ell,n}(x)\right] \longrightarrow \Upsilon_{2\ell}(x)  \int K^{2}(t) dt  \quad \quad \text{as} \quad n\rightarrow \infty.
\end{equation*}
In addition under \hyperref[d]{\bf{\textit{D}}}\textit{ii)} and \hyperref[k]{\bf{\textit{K}}}\textit{iii)}, we get easily
\begin{equation*}
\begin{aligned}
Cov(\Gamma_{1,n}(x),\Gamma_{2,n}(x))
&=\mathbb{E}[\Gamma_{1,n}(x)\Gamma_{2,n}(x)]-\mathbb{E}[\Gamma_{1,n}(x)]\mathbb{E}[\Gamma_{2,n}(x)]\\
&=h_n^{-1}\left\{ \mathbb{E}\left[\frac{\delta_{1}Y_1^{-3}}{\overline{G}^2(Y_1)}K^2\left(\frac{x-X_i}{h_n}\right)\right]-\mathbb{E}\left[\frac{\delta_{1} Y_{1}^{-1}}{\overline{G}(Y_1)} K\left(\frac{x-X_1}{h_n}\right)\right]\mathbb{E}\left[\frac{\delta_{1} Y_{1}^{-2}}{\overline{G}(Y_1)} K\left(\frac{x-X_1}{h_n}\right)\right]\right\}\\
&=\Upsilon_{3}(x)\kappa+o(1).
\end{aligned}
\end{equation*}
Next, we will show that any linear combinations are asymptotically gaussian. For let $(z_1, z_2)^t$ be a real numbers, we put
\begin{equation}\label{berry}
\Delta_n(x)=\sum_{\ell=1}^{2} z_\ell \Gamma_{\ell,n}(x)=:\sum_{i=1}^{n}\Big( z_1\Delta^1_{i,n}(x)+z_2\Delta^2_{i,n}(x)\Big)
\end{equation} 
where 
\[\Delta^\ell_{i,n}(x):= (n h_n)^{-1/2} \left\{ \frac{\delta_{i} Y_{i}^{-\ell}}{\overline{G}(Y_i)} K\left(\frac{x-X_i}{h_n}\right) - \mathbb{E}\left[ \frac{\delta_{i} Y_{i}^{-\ell}}{\overline{G}(Y_i)} K\left(\frac{x-X_i}{h_n}\right) \right] \right\}\; \textrm{for}\; \ell=1,2.\]  
Now in order to show that (\ref{berry}) is asymptotically normal we verify the Berry-Ess\`een condition (\hyperref[chow]{Chow and Teicher (1997), p. 322}). For that we need to prove : 
\begin{equation}\label{prop}
\rho^{ 3}_{n}:=\sum_{i=1}^{n}\mathbb{E}\left[\left|  \Delta^\ell_{i,n}(x)  \right|^{3}\right] \longrightarrow 0 
\end{equation}
with 
\[\mathbb{E}\left[\left|  \Delta\ell_{i,n}(x)  \right|^{3}\right] = \Big(\frac{h_n}{n}\Big)^{3/2} \mathbb{E}\Big[ \left|  \frac{\delta_{i} Y_{i}^{-\ell}}{h_n\overline{G}(Y_i)} K\left(\frac{x-X_i}{h_n}\right) - \mathbb{E}\left[ \frac{\delta_{i} Y_{i}^{-\ell}}{h_n\overline{G}(Y_i)} K\left(\frac{x-X_i}{h_n}\right) \right]\right|^3\Big].\]
Applying the $C_{r}$ inequality (see \hyperref[loeve] {Lo\`eve (1963), p. 155}), we get 
\[\begin{aligned}
\mathbb{E}\left[\left|  \Delta^\ell_{i,n}(x)  \right|^{3}\right]& \leq 4 \Big(\frac{h_n}{n}\Big)^{3/2} \left\{ \mathbb{E}\left[ \frac{\delta_{i} |Y_{i}|^{-3\ell}}{h_n^3\overline{G}^{3}(Y_i)} K^{3}\left(\frac{x-X_i}{h_n}\right) \right] +\left|\mathbb{E}\left[ \frac{\delta_{i} Y_{i}^{-\ell}}{h_n\overline{G}(Y_i)} K\left(\frac{x-X_i}{h_n}\right) \right]\right|^3 \right\}\\ 
& \leq 4 \Big(\frac{h_n}{n}\Big)^{3/2}\left\{ \mathcal{M}_{1}+\mathcal{M}_{2} \right\}.
\end{aligned}\]
and as $\mathcal{M}_{1}$ and $\mathcal{M}_{2}$ are bounded under \hyperref[k]{\bf{\textit{K}}}  which gives that $\rho^{ 3}_{n}=O\displaystyle \Big(\Big(\frac{h_n^{3}}{n}\Big)^{1/2}\Big)=o(1)$. Now, under \hyperref[h]{\bf{\textit{H}}}\textit{i)} the property (\ref{prop}) is satisfied which proves the asymptotic normality of $\Gamma_{\ell,n}(x)$  and, together with (\ref{gamma}) and (\ref{sigma}), complete the proof of Lemma \ref{var_cov}.\\
\end{proof}
Now to complete the proof of Theorem \ref{theo2}, consider the mapping $\theta$ from $\mathbb{R}\times\mathbb{R}^*_+$ to $\mathbb{R}$ defined by $\theta(x,y)=x/y$. We deduce from Mann-Wald's Theorem (see \hyperref[rao] {Rao 1965, p. 321}) that:
\[ \sqrt{nh_n}(r_n(x)-r(x)) \xrightarrow{\  \mathcal{D}    \ } \mathcal{N}(0,\kappa \nabla\theta^T\Sigma(x) \nabla\theta)\]
where the gradient $\nabla\theta^T=\displaystyle\left(\frac{\partial\theta}{\partial x},\frac{\partial\theta}{\partial y}\right)$ is evaluated at $(\bar{r}_1(x),\bar{r}_2(x))$. Simple algebra gives then the variance
\[\sigma^2(x)=\kappa\frac{\Upsilon_2(x)\bar{r}_2^2(x)-2\Upsilon_3(x)\bar{r}_1(x)\bar{r}_2(x)+\Upsilon_4(x)\bar{r}_1^2(x)}{\bar{r}_2^4(x)},\] which completes the proof.
\end{proof}
\addcontentsline{toc}{chapter}{References}
\bibliographystyle{alpha}

\end{document}